\documentclass[onefignum,onetabnum]{siamart171218}
\usepackage{algorithmic}
\usepackage{amsmath}
\usepackage{amsfonts}
\usepackage{enumitem}
\usepackage{graphicx}
\graphicspath{{img/}}
\usepackage[utf8x]{inputenc}
\usepackage{xcolor}
\usepackage[caption=false]{subfig}
\usepackage{placeins}

\usepackage{resizegather}

\newcommand{\R}{\mathbb{R}}
\newcommand{\N}{\mathcal{N}}

\newcommand{\X}{\mathcal{X}}
\newcommand{\Y}{\mathcal{Y}}

\newcommand{\norm}[1]{\left\| #1 \right\|}
\newcommand{\sigmafunc}[1]{\sigma \left( #1 \right)}

\newtheorem{assumption}{Assumption}
\newtheorem{example}{Example}

\newsiamremark{remark}{Remark}

\headers{A Unifying Framework for Sparsity Constrained  Optimization}{M. Lapucci, T. Levato, F. Rinaldi, M. Sciandrone}

\title{A Unifying Framework for Sparsity Constrained  Optimization\thanks{Submitted to the editors DATE.}}

\author{Matteo Lapucci\thanks{Dipartimento di Ingegneria dell'Informazione, Universit\`a di Firenze, Via di Santa Marta 3, 50139 Firenze, Italy
		(\email{matteo.lapucci@unifi.it},
		\email{tommaso.levato@unifi.it}, \email{marco.sciandrone@unifi.it}).}
	\and Tommaso Levato\footnotemark[2]
	\and Francesco Rinaldi\thanks{Dipartimento di Matematica "Tullio Levi-Civita", Universit\`a di Padova, Via Trieste 63, 35121 Padova, Italy (\email{rinaldi@math.unipd.it}).}
	\and Marco Sciandrone\footnotemark[2]}

\begin{document}
\maketitle	

\begin{abstract}
	In this paper, we consider the optimization problem of minimizing a continuously differentiable function subject to both convex constraints and sparsity constraints. By exploiting a mixed-integer reformulation from the literature, we define a necessary optimality condition based on a tailored neighborhood that allows to take into account potential changes of the support set. We then propose an algorithmic framework to tackle the considered class of problems and prove its convergence to points satisfying the newly introduced concept of stationarity. We further show that, by suitably choosing the neighborhood, other well-known optimality conditions from the literature can be recovered at the limit points of the sequence produced by the algorithm.
	Finally, we analyze the computational impact of  the neighborhood size within our framework and  in the comparison with some state-of-the-art algorithms, namely, the Penalty Decomposition method and the Greedy Sparse-Simplex method. The algorithms have been tested using a benchmark related to sparse logistic regression problems.
\end{abstract}


	
\section{Introduction}
We consider the following sparsity constrained problem: 
\begin{equation}
\label{prob: cc_prob}
\begin{array}{cl}
\displaystyle \min_x & f(x) \\
\text{s.t.} & \| x \|_0 \leq s, \\
			& x \in X, 
\end{array}
\end{equation}
where $f:\R^n \to \R$ is a continuously differentiable function, $X \subseteq \R^n$ is a closed and convex set,
and $s<n$ is a properly chosen integer value. We further use $\X$ to indicate the overall feasible set $X\cap\{x\in\R^n\mid \|x\|_0\le s\}$.

Problem \eqref{prob: cc_prob} has a wide range of applications,  from subset selection in regression \cite{Miller1984} 
and the compressed sensing technique used in signal processing \cite{Candes2008} to portfolio optimization \cite{Bienstock1996,Mutunge2018}. 
Such a problem  can be reformulated into equivalent different mixed-integer problems and is known to be $\mathcal{NP}$-hard \cite{Bienstock1996,Natarajan1995,Shaw2008}.

The approaches proposed in the literature for the solution of problem \eqref{prob: cc_prob} include: exact methods (see, e.g., \cite{Bertsimas2009,Bienstock1996,Shaw2008,Vielma2008})
typically based on  branch-and-bound or branch-and-cut strategies; methods that handle suitable reformulations
of the problem based on orthogonality
constraints (see, e.g., \cite{Branda2018,bucher2018second,Burdakov2016,vcervinka2016constraint}); penalty decomposition methods, where 
penalty subproblems are solved by a block coordinate descent method \cite{lapucci2021convergent,Lu2013}; methods that identify points satisfying tailored optimality conditions related to the problem \cite{Beck2013,Beck2016}; heuristics like evolutionary algorithms \cite{Anagnostopoulos2010}, particle swarm methods \cite{Boudt2019,Deng2012}, 
genetic algorithms, tabu search and simulated annealing \cite{Chang2000}, and also neural networks \cite{Fernandez2007}.

We observe that problem \eqref{prob: cc_prob} is generally hard to solve because both the objective function and the feasible set
(due to the combinatorial nature of the sparsity constraint) are nonconvex. The inherently combinatorial flavor of the given problem  makes the definition of proper optimality conditions and, consequently, the development of algorithms that generate points satisfying those conditions a challenging task. 
A number of ways to address these issues are proposed in the literature (see, e.g., \cite{Beck2013,Beck2016,Burdakov2016,lapucci2021convergent,Lu2013}). However, some of the optimality conditions proposed do not fully take into account the combinatorial nature of the problem, whereas some of the corresponding algorithms \cite{Beck2013,Lu2013} require to exactly solve a sequence of nonconvex subproblems and this may be practically prohibitive. Moreover, due to the theoretical tools involved in the analysis, it is anyway not easy to relate  the different approaches with each other.

In this paper, we hence  give a unifying view on this matter. More specifically,  we consider the mixed-integer reformulation of problem \eqref{prob: cc_prob} described in \cite{Burdakov2016} and use 
it to define a suitable optimality condition.  
This condition is then embedded into an algorithmic framework aimed at finding points satisfying the resulting optimality criterion. The algorithm combines inexact minimizations with a strategy that explores tailored neighborhoods of a given feasible point. Those features make it easy to handle the nonconvexity in both the objective function  and the feasible set also from a practical  point of view.
We prove the convergence of the algorithmic scheme, establishing that its limit points satisfy the specific optimality condition. We then show that different  conditions proposed in the literature (see, e.g., \cite{Beck2013,Burdakov2016,Lu2013}) can be easily derived from ours. We finally perform some numerical tests on sparse logistic regression in order to show that the devised method is also computationally viable. 

\par\smallskip\noindent

The paper is organized as follows: in \cref{sec:opt_cond}, we provide basic definitions and preliminary results related to optimality conditions of problem \eqref{prob: cc_prob}.  In \cref{sec:framework}, we describe our proposed algorithmic framework and show (\cref{sec:convergence}) the convergence analysis without constraint qualifications. In \cref{sec:convergenceKKT}, we analyze the asymptotic convergence properties of the algorithm when constraint qualifications hold. Finally, we report  numerical experiments in \cref{sec:experiments} and give some concluding remarks in \cref{sec:conclusions}. We also provide in \cref{sec:appendix} some insights on the relationship between classical stationarity conditions for convex problems with and without constraints qualifications.

 
\section{Basic definitions and preliminary results}
\label{sec:opt_cond}
Even though problem \eqref{prob: cc_prob} is a continuous optimization problem, it has an intrinsic combinatorial nature and in applications the interest often lies in finding a good, possibly globally optimal configuration of active variables. Being \eqref{prob: cc_prob} a continuous problem, $x^*\in \X$ is a local minimizer if there exists an open ball $\mathcal{B}(x^*,\epsilon)$ such that $f(x^*)=\min\{f(x)\mid x\in \X\cap \mathcal{B}(x^*,\epsilon)\}$. In some works from the literature (e.g., \cite{Burdakov2016,Lu2013}) necessary conditions of local optimality have been proposed. However, for this particular problem every local minimizer for a fixed active set of $s$ variables is a local minimizer of the given problem. Hence the number of local minimizers grows as fast as $\binom{n}{s}$ and is thus of low practical usefulness.

In \cite{Beck2013,Beck2016}, the authors propose necessary conditions for global optimality that go beyond the concept of local minimum described above, thus allowing to consider possible changes to the structure of the support set, and reducing the pool of optimal candidates. However, these conditions are either tailored to the ``unconstrained case'', or limited to moderate changes in the support, or involve hard operations, such as exact minimizations or projections onto nonconvex sets. 

In order to introduce a general and affordable necessary optimality condition that also takes into account the combinatorial nature of the problem, we consider in our analysis 
the equivalent reformulation
of problem 
\eqref{prob: cc_prob} described in
\cite{Burdakov2016}:
\begin{equation}\label{prob:mi_prob}
	\begin{aligned}
	\min_{x,y}\;&f(x)\\
		 \text{s.t. }& e^\top y \geq n - s, \\
		 & x_i y_i = 0,  \quad \forall\, i=1,\ldots,n,\\
		 & x \in X,\\
		 & y \in \{ 0, 1 \}^n.
	\end{aligned}
\end{equation}
From here onwards, we will use the following notation:
\begin{equation*}
\arraycolsep=1.4pt
\begin{array}{rl}
\Y 		&= \{y \mid y \in \{0,1\}^n, e^\top y \geq n - s\},\\
\X(y) &= \{ x \in X \mid x_i y_i = 0 \ \forall\, i=1,\ldots,n \}.
\end{array}
\end{equation*}
We further define the support set of a vector $z$  by
\[ I_1(z) = \{ i \mid z_i \neq 0 \}, \]
while its complement is defined by
\[ I_0(z) = \{ i \mid z_i = 0 \}. \]
Moreover, we recall the concept of super support set \cite{Beck2016}:
\begin{definition} Let us consider a feasible point $z$ for problem \eqref{prob: cc_prob}. A set $J\subset \{1,\ldots ,n\}$ is called super support of $z$ if it is such that $|J|=s$ and $I_1(z)\subseteq J$.
\end{definition}
We denote by $z_I$ the subvector of $z$ identified by the components contained in an index set $I$. We also denote by $\Pi_C$ the orthogonal projection operator over the closed convex set $C$. We notice that given a feasible point $(x,y)$ of problem \eqref{prob:mi_prob}, the components $I_0(y)$ give an \emph{active subspace} for $x$, i.e., those components identify the subspace where the nonzero components of $x$ lay. We thus have that $I_1(x)\subseteq I_0(y)$.

Nonlinear mixed-integer programs have been characterized exploiting the notion of \textit{neighborhood} \cite{li2006nonlinear,Lucidi2005}. Given a feasible point $(x, y)$, a discrete neighborhood $\N (x,y)$ is a set of feasible points that are close, to some extent, to $(x,y)$ and that contains $(x,y)$ itself. 

We introduce here an example of tailored neighborhood for problem \eqref{prob:mi_prob} that
can be implemented at a reasonable computational cost. Such a neighborhood will also help us to relate our analysis to the  other theoretical tools available in the literature.
\begin{definition}
	\label{ex:discrete_neighborhood}
	Let $d_H:\{0,1\}^n\times\{0,1\}^n\to \mathbb{N}$ denote the Hamming distance. Moreover, let $J(y,\hat{y})=\{i\mid y_i\neq \hat{y}_i\}$ and let ${H}_{J(y,\hat{y})}(\cdot)$ be a function such that $\hat{x}={H}_{J(y,\hat{y})}(x)$ is defined as
	$$\left\{\!\!\begin{array}{ll}
	\hat{x}_h= 0 &\text{if }h\in J(y,\hat{y})\\
	\hat{x}_h=x_h&\text{otherwise}
	\end{array}\right.$$
	Then, given $\rho\in\mathbb{N}$, the  neighborhood is
	\begin{equation}
	\label{D_N}
	\mathcal{N}_\rho(x,y)=\left\{(\hat{x},\hat{y})\mid e^\top\hat{y}\ge n-s,\;d_H(\hat{y},y)\le \rho,\; \hat{x}=H_{J(y,\hat{y})}(x)\right\}.
	\end{equation}
\end{definition}
We notice that this particular definition of neighborhood allows to take into account the potential ``change of status'' of up to $\rho$ variables in the vector $\hat y$ defining an active subspace.
\begin{example}
\label{ex:discrete_neighborhood_hamming}
Consider the problem  \eqref{prob:mi_prob}
with $n=3$ and $s=2$ and let $\rho = 2$.
Let $(x,y)$ be a feasible point defined as follows
$$
(x,y)=
\left(
\begin{array}{l}
1\\
2\\
0
\end{array}
\right)
\left(
\begin{array}{l}
0\\
0\\
1
\end{array}
\right)
$$
The neighborhood ${\cal N}_\rho(x,y)$ is given by
\begin{gather*}
    {\cal N}_2(x,y)=\left\{
\left(\begin{array}{l}
1\\2\\0
\end{array}\right)
\left(\begin{array}{l}
0\\0\\1
\end{array}\right),
\;
\left(\begin{array}{l}
1\\0\\0
\end{array}\right)
\left(\begin{array}{l}
0\\1\\0
\end{array}\right),
\;
\left(\begin{array}{l}
0\\2\\0
\end{array}\right)
\left(\begin{array}{l}
1\\0\\0
\end{array}\right),
\;
\left(\begin{array}{l}
1\\0\\0
\end{array}\right)
\left(\begin{array}{l}
0\\1\\1
\end{array}\right),
\;
\left(\begin{array}{l}
0\\2\\0
\end{array}\right)
\left(\begin{array}{l}
1\\0\\1
\end{array}\right),
\;
\left(\begin{array}{l}
0\\0\\0
\end{array}\right)
\left(\begin{array}{l}
1\\1\\1
\end{array}\right)
\right\}
\end{gather*}
\end{example}

Now, a notion of local optimality for problem \eqref{prob:mi_prob}, depending on the neighborhood $\N (x,y)$, can be introduced: 
\begin{definition}
	A point $(x^*,y^*)\in \X(y^*)\times \Y$ is a local minimizer of problem \eqref{prob:mi_prob} if there exists an $\epsilon> 0$ such that and for all $(\hat{x},\hat{y})\in\mathcal{N}(x^*,y^*)$ it holds
	\begin{equation*}
		f(x^*)\le f(x)\quad \forall\, x \in \mathcal{B}(\hat{x},\epsilon )\cap X(\hat{y}).
	\end{equation*}
\end{definition}

Note that in the above definition the continuous nature of the problem, expressed by the variables $x$, is taken into account by means of the standard ball $\mathcal{B} (\hat{x}, \epsilon)$. The given definition clearly depends on the choice of the discrete neighborhoods. A larger neighborhood $\N (x^*, y^*)$ should give a better local minimizer, but the  computational effort needed to locate the solution may increase.

Inspired by the definition of local optimality for problem \eqref{prob:mi_prob}, we introduce a necessary condition of global optimality for problem \eqref{prob: cc_prob} that allows to take into account possible, beneficial changes of the support and that hence properly captures, from an applied point of view, the essence of the problem.  

Such a condition relies on the use of stationary points related to continuous problems obtained by fixing the binary variables in problem \eqref{prob:mi_prob}, i.e., for a fixed $\bar y \in \Y$,
		\begin{equation}\label{subprob_cont}
			\begin{aligned}
			\min \;& f(x) \\
			\text{s.t. } & x \in \mathcal{X}(\bar y).
			\end{aligned}
		\end{equation}
\begin{definition}
	\label{def: stationary}
	A point $x^*\in {\cal X}$ is called an 
	$\N$-stationary point, if there exists an $y^*\in {\cal Y}$
	such that  
	\begin{enumerate}[label=(\roman*)]
	  \item $(x^*, y^*)$ is feasible for problem \eqref{prob:mi_prob};
	  \item the point $x^*$ is a stationary point of the continuous problem
		\begin{equation*}
			\begin{aligned}
			\min \;& f(x) \\
			\text{s.t. } & x \in \mathcal{X}(y^*);
			\end{aligned}
		\end{equation*}
		\item every $(\hat{x}, \hat{y}) \in  \N_\rho(x^*, y^*)$ satisfies $f(\hat{x}) \geq f(x^*)$ and  if $f(\hat{x}) = f(x^*)$, the point $\hat{x}$ is a \mbox{stationary} point of the continuous problem
		\begin{equation*}
		\begin{aligned}
		\min \;& f(x) \\
		\text{s.t. } & x \in \mathcal{X}(\hat{y}).
		\end{aligned}
		\end{equation*}
	\end{enumerate}
\end{definition}
It is easy to see that the following result stands:
\begin{theorem}\label{th:nec_cond}  
Let $x^*$ be a minimum point of problem \eqref{prob: cc_prob}. Then $x^*$ is an $\N$-stationary point.
\end{theorem}

We will show later in this work that the definition of $\N$-stationariy allows to retrieve in a unified view most of the known optimality conditions, if a suitable neighborhood $\N$ is employed.

In \cref{def: stationary} we generically refer to stationary points of problem \eqref{subprob_cont}, namely, to points
satisfying suitable optimality conditions. Then, concerning the assumptions on the feasible set $\mathcal{X}(\bar y)$, we distinguish the two cases:
\begin{itemize}
 \item [(i)] no constraint qualifications hold;
 \item [(ii)] constraint qualifications are satisfied
 and the usual KKT theory can be applied.
\end{itemize}
In case (i), we will refer to the following definition of stationary point of problem \eqref{subprob_cont}.
\begin{definition}\label{stat_1}
Given $\bar y \in \Y$ and $\bar x\in \mathcal{X}(\bar y)$, we say that $\bar x$ is a stationary point of problem \eqref{subprob_cont} if and only if
$$
\bar x=\Pi_{\mathcal{X}(\bar y)}\left[\bar x-\nabla f(\bar x)\right].$$
\end{definition}
We notice that $\mathcal{X}({\bar y})$ is a convex set when $X$ is convex, then the condition given above is a classic stationarity condition for the problem \eqref{subprob_cont}. Case (ii) will be considered later.

\section{Algorithmic framework}
\label{sec:framework}
Here, we discuss an algorithmic framework for the solution of problem \cref{prob: cc_prob} that exploits the reformulation given in problem \cref{prob:mi_prob}.
The proposed approach is somehow related to classic methods for mixed variable programming proposed in the literature (see, e.g.,  \cite{li2006nonlinear,Lucidi2005}).
Roughly speaking, the approach is based at each iteration on the definition
of a suitable neighborhood $\N(x^k, y^k)$ of the current point $(x^k,y^k)$ and on exploratory moves with respect to the continuous variables around the points of the neighborhood.

Concerning the exploration move, it is a local search performed by an Armijo-type line search along the projected gradient direction. The procedure is formalized in \cref{alg:PGLS}.

For any point $(\hat{x}^k, \hat{y}^k) \in \N(\tilde{x}^k, y^k)$ that is not significantly worse (in terms of the objective value) than the current candidate, we perform a local continuous search around $\hat{x}^k$; we skip to the following iteration as soon as a point providing a sufficient decrease of the objective value is found. The algorithm, which we refer to as Sparse Neighborhood Search (SNS) is formally defined in \cref{alg:MISO}.

\begin{algorithm}
	\caption{Projected-Gradient Line Search (PGLS)}
	\label{alg:PGLS}
	{\bf input}: $y \in \Y, x \in \X(y), \gamma \in (0, \frac{1}{2}), \delta \in (0,1), \alpha=1$.\\
	{\bf Step 1}: Set $\hat{x} = \Pi_{\X(y)} \left[x - \nabla f(x)\right]$, $d = \hat{x} - x$.\\
	{\bf Step 2}: If
	\[ f(x + \alpha d ) \leq f(x) + \gamma \alpha \nabla f(x)^\top d, \]
	set $\tilde{x} = x + \alpha d$ and exit.\\
	{\bf Step 3}: Set $\alpha = \delta \alpha$ and go to Step 2.
\end{algorithm}

\begin{algorithm}[H]
	\caption{Sparse Neighborhood Search (SNS)}
	\label{alg:MISO}
	\begin{algorithmic}
		\STATE{\textbf{input:} $y^0 \in \Y, x^0 \in \X(y^0), \xi \geq 0, \theta \in (0,1), \eta_0 > 0, \mu_0 > 0, \delta \in (0, 1)$.}
		\STATE{\textbf{Step 0:} Set $k = 0$.}
		\STATE{\textbf{Step 1:} Compute $\tilde{x}^k$ by PGLS($x^k, y^k$).}
		\STATE{\textbf{Step 2:} Define $W_k = \{ (x,y) \in \N(\tilde{x}^k, y^k) \mid f(x) \leq f(\tilde{x}^k) + \xi \}$.}
		\begin{ALC@g}
			\STATE{\textbf{2.1:} If $W_k \neq \emptyset$, choose $(x^\prime, y^\prime) \in W_k$, set $j=1, x^j=x^\prime$. Otherwise, go to Step 3.\\}
			\STATE{\textbf{2.2:} Compute $x^{j+1}$ by PGLS($x^j, y^\prime$).}
			\STATE{\textbf{2.3:} If $f(x^{j+1}) \leq f(\tilde{x}^k) - \eta_k$, set $x^{k+1} = x^{j+1}, y^{k+1}=y^\prime, \eta_{k+1}=\eta_k$ and go to Step 4.}
			\STATE{\textbf{2.4:} If $\norm { x^j - \Pi_{\X(y^\prime)} \left[x^j - \nabla f(x^j)\right] } > \norm{ x^k - \Pi_{\X(y^k)}\left[x^k - \nabla f(x^k)\right] } + \mu_k$, set $j=j+1$ and go to 2.2. Otherwise, set $W_k = W_k \setminus \{ (x^\prime, y^\prime) \}$ and go to 2.1.}
		\end{ALC@g}
		\STATE{\textbf{Step 3:} Set $x^{k+1} = \tilde{x}^k, y^{k+1}=y^k$. If $f(x^{k+1}) \leq f(x^k) - \eta_k$, set $\eta_{k+1}=\eta_k$. Otherwise set $\eta_{k+1}=\theta \eta_k$.}
		\STATE{\textbf{Step 4:} Set $\mu_{k+1} = \delta \mu_k$, $k=k+1$ and go to Step 1. }
	\end{algorithmic}
\end{algorithm}

\subsection{Convergence analysis}
\label{sec:convergence}
In this section, we prove a set of results concerning the properties of the sequences produced by \cref{alg:MISO}. Note that in this Section we employ the concept of stationarity \cref{eq:proj-stat}. First, we state some suitable assumptions.

\begin{assumption}
	The gradient $\nabla f(x)$ is Lipschitz-continuous, i.e., there exists a constant $L > 0$ such that
	\begin{equation*}
	\norm{ \nabla f(x) - \nabla f(\bar{x}) } \leq L \norm{ x - \bar{x} }
	\end{equation*}
	for all $x, \bar{x} \in \R^n$.
\end{assumption}

\begin{assumption}
	\label{ass:level_set_compact}
	Given $y^0 \in \Y$, $x^0 \in \X(y^0)$ and a scalar $\xi > 0$, the level set
	\[ \mathcal{L}(x^0, y^0) = \{ (x,y) \in \X(y) \times \Y \mid f(x) \leq f(x^0) + \xi \} \]
	is compact.
\end{assumption}
The crucial point in the proposed framework is choosing suitable discrete neighborhoods. 
First, note that when we deal with both continuous and integer variables, the usual notion of convergence to a point needs to be tweaked. 
In particular, we have the following definition.

\begin{definition}
A sequence $\{ (x^k, y^k) \}$ converges to a point $(\bar{x}, \bar{y})$ if for any $\epsilon > 0$ 
there exists an index $k_\epsilon$ such that for all $k \geq k_\epsilon$ we have that $y^k = \bar{y}$ and $\| x^k - \bar{x} \| < \epsilon$.
\end{definition}

To ensure convergence to meaningful points, we need a ``continuity'' assumption on the discrete neighborhoods we explore.

\begin{assumption}
\label{asmp:discrete_neighborhood}
Let $\{ (x^k, y^k) \}$ be a sequence converging to $(\bar{x}, \bar{y})$. Then, for any $(\hat{x}, \hat{y}) \in \N(\bar{x}, \bar{y})$, 
there exists a sequence $\{ (\hat{x}^k, \hat{y}^k) \}$ converging to $(\hat{x}, \hat{y})$ such that $(\hat{x}^k, \hat{y}^k) \in \N(x^k, y^k)$.
\end{assumption}
The assumption above is a mild continuity assumption on the discrete neighborhoods and is equivalent 
to the lower semicontinuity of a point-to-set function as defined in \cite{Berge1963}. 
Next, we properly define the discrete neighborhood used in our algorithmic framework.
 
Now, a discrete neighborhood, by definition, is a set of feasible points. 
In the case when $\X \subset \R^n$, zeroing variables may result in points that are not feasible. 
For this reason, we initially consider an easier version of problem \cref{prob:mi_prob} where $\X = \R^n$. In this case the neighborhood $\N_\rho$ defined in \eqref{D_N} contains feasible points. Moreover, it satisfies \cref{asmp:discrete_neighborhood}, as stated here below. 

\begin{proposition}
\label{prop:discrete_neighborhoods_no_X}
The point-to-set map $\N_\rho(x,y)$  defined in \cref{ex:discrete_neighborhood} satisfies \cref{asmp:discrete_neighborhood}.
\end{proposition}
\begin{proof}
	Let $\{x^k,y^k\}$ be a sequence convergent to $\{\bar{x},\bar{y}\}$. Then, for any $\epsilon>0$, there exists $k_\epsilon$ such that $y^k=\bar{y}$ and $\|x^k-\bar{x}\|\le \epsilon$ for all $k>k_\epsilon$. Let $(\hat{x},\hat{y})\in\mathcal{N}_\rho(\bar{x},\bar{y})$.
	Since $y^k=\bar{y}$ for $k$ sufficiently large, $\{y\mid e^{\top}y\ge n-s,\; d_H(y,y^k)\le \rho\}=\{y\mid e^{\top}y\ge n-s,\; d_H(y,\bar{y})\le \rho\}$, hence $\hat{y}\in\{y\mid d_H(y,y^k)\le \rho\}$ for all $k$.
	
	Let us then consider the sequence $\{\hat{x}^k,\hat{y}^k\}$ where $\hat{y}^k=\hat{y}$ and $\hat{x}^k=H_{J(y^k,\hat{y})}(x^k)$. We can observe that $(\hat{x}^k,\hat{y}^k)\in\mathcal{N}_\rho(x^k,y^k)$. Now, let $j\in\{1,\ldots,n\}$. The set $J(y^k,\hat{y}^k)=J(\bar{y},\hat{y})=J$ is constant for $k$ sufficiently large.
	
	If $j\notin J$, we have
	$$\lim_{k\to \infty}\hat{x}^k_j = \lim_{k\to\infty}x^k_j=\bar{x}_j = \hat{x}_j.$$ 
	On the other hand, if $j\in J$, $\hat{x}^k_j=0$ and $\hat{x}_j=0$. Hence $$\lim_{k\to \infty}\hat{x}^k=\hat{x}$$
	and we thus get the thesis.
\end{proof}

To generalize the previous proposition to the case where $\X \subset \R^n$, we can replace each $(\tilde{x}, \tilde{y}) \in \mathcal{N}_\rho(x,y)$ with the point $(\hat{x}, \hat{y})$, where $\hat{y} = \tilde{y}$ and $\hat{x} = \Pi_{\X(\hat{y})}(\tilde{x})$. In other words, first we change the structure of the active set, then we project the $x$ part onto $\X(\hat{y})$, which is a convex set. In the following, we will refer to this new discrete neighborhood with $\N_{\mathcal{C}\rho}(x,y)$.

\begin{proposition}
\label{prop:discrete_neighborhoods_X}
Let $\{(x^k, y^k)\}$ be a sequence converging to $(\bar{x}, \bar{y})$. Then, the neighborhood $\N_{\mathcal{C}\rho}(\bar{x}, \bar{y})$ satisfies \cref{asmp:discrete_neighborhood}.
\end{proposition}
\begin{proof}
The proof follows exactly as in \cref{prop:discrete_neighborhoods_no_X}, recalling the continuity of the projection operator $\Pi_{\X(\hat{y})}$.
\end{proof}

Before turning to the convergence analysis of the algorithm, we prove a further useful preliminary result concerning the neighborhood $\N_\rho$. Notice that this result can be easily extended to $\N_{\mathcal{C}\rho}$. In order to avoid getting a too much cumbersome notation, we will always refer to $\N_\rho$ from now on, even when dealing with additional constraints.
\begin{lemma}\label{prel_lemma}
	Let $y\in \mathcal{Y}$ and $x\in \mathcal{X}(y)$ with $\delta=\|x\|_0$. Let us consider the set
	\begin{equation*}
	\bar{\cal N}(x)=\{(\hat x,\hat y)\mid \hat x = x,\;e^\top\hat y=n-s,\; I_0(\hat y)\supseteq I_1(x)\; \}.
	\end{equation*}
	We have that $$\bar{\cal N}(x)\subseteq\mathcal{N}_\rho(x,y),$$ when $\rho\geq2(s-\delta)$.
\end{lemma}
\begin{proof}
	Let $(\hat x, \hat y)$ be any point in $\bar{\cal N}(x)$. From the feasibility of $(x,y)$ we have 
	\begin{equation}
	\label{eq:from_y_def}
	\delta\le I_0(y)\le s\qquad n-s\le I_1(y)\le n-\delta.     
	\end{equation}
	Moreover, from the definition of $\bar{\cal N}(x)$, we have $$I_0(\hat{y})=s\qquad I_1(\hat{y}) = n-s.$$
	
	Now, it is easy to see that 
	\begin{equation}
	\label{eq:hamming_remark}
	d_H(y,\hat{y}) = n - |I_0(y) \cap I_0(\hat{y})| - |I_1(y)\cap I_1(\hat{y})|.
	\end{equation}
	We can note that, since $I_0(y)\supseteq I_1(x)$ and $I_0(\hat{y})\supseteq I_1(x)$, it has to be $I_0(y)\cap I_0(\hat{y}) \supseteq I_1(x)$.
	Therefore 
	\begin{equation}
	\label{eq:supseteq}
	|I_0(y) \cap I_0(\hat{y})| \ge |I_1(x)| = \delta.
	\end{equation}
	We can now turn to $I_1(y)\cap I_1(\hat{y})$. Since the latter set can be equivalently written, by De Morgan's law, as $\{1,\ldots,n\}\setminus (I_0(y)\cup I_0(\hat{y}))$, we can obtain
	\begin{equation*}
	\begin{aligned}
	|I_1(y)\cap I_1(\hat{y})| &= |\{1,\ldots,n\}\setminus (I_0(y)\cup I_0(\hat{y}))|\\ &= n - |I_0(y)\cup I_0(\hat{y})|\\&
	=n - (|I_0(y)| + |I_0(\hat{y})| - |I_0(y)\cap I_0(\hat{y})|)\\&
	= n - |I_0(y)| - s + |I_0(y)\cap I_0(\hat{y})|\\
	&\ge n - s - s + \delta\\&= n -2s + \delta,
	\end{aligned}
	\end{equation*}
	where the second last inequality comes from \eqref{eq:from_y_def} and \eqref{eq:supseteq}. Putting everything together back in \eqref{eq:hamming_remark}, we get
	$$d_H(y,\hat{y}) \le n - \delta - n+2s - \delta = 2(s-\delta).$$
	Taking into account that $\rho\geq2(s-\delta)$ in 
	the definition of $\mathcal{N}_\rho(x,y)$, we obtain 
	$$(\hat x, \hat y)\in \mathcal{N}_\rho(x,y),$$
	thus getting the desired result.
\end{proof}
\vspace{1em}

We can now focus on the algorithms. First, we prove a property of \cref{alg:PGLS} that will play an important role in the convergence analysis of \cref{alg:MISO}.
\begin{proposition}
	\label{prop:property_A}
	Given a feasible point $(x, y)$, \cref{alg:PGLS} produces a feasible point $(\tilde{x}, y)$ such that
	\[ f(\tilde{x}) \leq f(x) - \sigmafunc{\norm{ x - \Pi_{\X(y)} \left[x - \nabla f(x)\right]}}, \]
	where the function $\sigmafunc{\cdot} \geq 0$ is such that if $\sigmafunc{t^h} \to 0$ then $t^h \to 0$.
\end{proposition}
\begin{proof}
	By definition, $d = \hat{x} - x$, where $\hat{x} = \Pi_{\X(y)} \left[x - \nabla f(x)\right]$. By the properties of the projection operator, we can write
	\begin{equation*}
	(x - \nabla f(x) - \hat{x})^\top (x - \hat{x}) \leq 0,
	\end{equation*}
	which, with simple manipulations, implies that
	\begin{equation}
	\label{eq: proj_prop}
	\nabla f(x)^\top d \leq - \norm{ d }^2 = - \norm{ x - \Pi_{\X(y)}\left[x - \nabla f(x)\right] }^2.
	\end{equation}
	
	By the instruction of the algorithm, either $\alpha=1$ or $\alpha < 1$.
	
	If $\alpha = 1$, then $\tilde{x} = x + d$ satisfies
	\begin{equation}
	\label{eq:armijo_condition_alpha_1}
	f(\tilde{x}) \leq f(x) + \gamma \nabla f(x)^\top d \leq f(x) - \gamma \norm{ x - \Pi_{\X(y)} \left[x - \nabla f(x)\right] }^2.
	\end{equation}
	
	If $\alpha < 1$, we must have that
	\begin{equation}
	\label{eq:armijo_condition}
	f(x + \alpha d) \leq f(x) + \gamma \alpha \nabla f(x)^\top d,
	\end{equation}
	\begin{equation}
	\label{eq:armijo_condition_unsatisfied}
	f\left(x + \frac{\alpha}{\delta} d\right) > f(x) + \gamma \frac{\alpha}{\delta} \nabla f(x)^\top d.
	\end{equation}
	Applying the mean value theorem to equation \cref{eq:armijo_condition_unsatisfied}, we get
	\begin{equation*}
	\nabla f\left(x + \theta \frac{\alpha}{\delta} d\right)^\top d > \gamma \nabla f(x)^\top d,
	\end{equation*}
	where $\theta \in (0,1)$.
	Adding and subtracting $\nabla f(x)^\top d$, and rearranging, we get
	\begin{equation*}
	(1 - \gamma) \nabla f(x)^\top d > \left[ \nabla f(x) - \nabla f\left(x + \theta \frac{\alpha}{\delta} d\right) \right]^\top d.
	\end{equation*}
	By the Lipschitz-continuity of $\nabla f(x)$, we can write
	\begin{equation*}
	\left[ \nabla f(x) - \nabla f\left(x + \theta \frac{\alpha}{\delta} d\right) \right]^\top d \geq -L \frac{\alpha}{\delta} \norm{ d }^2,
	\end{equation*}
	which means that
	\begin{equation*}
	(1 - \gamma) \nabla f(x)^\top d > -L \frac{\alpha}{\delta} \norm{ d }^2,
	\end{equation*}
	Rearranging, we get
	\begin{equation*}
	\frac{\delta}{L}(1 - \gamma) \nabla f(x)^\top d > -\alpha \norm{ d }^2.
	\end{equation*}
	This last inequality, together with \cref{eq: proj_prop}, yields
	\begin{equation*}
	\frac{\delta}{L}(1 - \gamma) \nabla f(x)^\top d > \alpha \nabla f(x)^\top d,
	\end{equation*}
	and substituting in equation \cref{eq:armijo_condition} we finally get
	\begin{equation*}
	f(\tilde{x}) < f(x) + \gamma \frac{\delta}{L}(1 - \gamma) \nabla f(x)^\top d \leq f(x) - \gamma \frac{\delta}{L}(1 - \gamma) \norm{ x - \Pi_{\X(y)} \left[x - \nabla f(x)\right] }^2.
	\end{equation*}
	This last inequality, together with \cref{eq:armijo_condition_alpha_1}, implies that 
	\[f(\tilde{x}) \leq f(x) - \sigmafunc{ \norm{ x - \Pi_{\X(y)} \left[x - \nabla f(x)\right] }} \]
	where
	\[ \sigmafunc{ t } = \gamma \min \left\{ 1,\frac{\delta}{L}(1-\gamma)  \right\} t^2. \]
\end{proof}
We can now state a couple of preliminary theoretical results.
We first show that \cref{alg:MISO} is well-posed.
\begin{proposition}
	For each iteration $k$, Step 2 of \cref{alg:MISO} terminates in a finite number of steps.
\end{proposition}
\begin{proof}
	Suppose by contradiction that Steps 2.1-2.4 generate an infinite loop, so that an infinite sequence of points $\{x^j\}$ is produced for which
	\begin{equation}
	\label{eq: contradiction}
	\norm{ x^j - \Pi_{\X(y^\prime)} \left[x^j - \nabla f(x^j)\right] } > \norm{ x^k - \Pi_{\X(y^k)} \left[x^k - \nabla f(x^k)\right] } + \mu_k > 0 \quad \forall j.
	\end{equation}
	By \cref{prop:property_A}, for each $j$ we have that
	\begin{equation}
	\label{eq:x_j+1-x_j}
	f(x^{j+1}) - f(x^j) \leq -\sigmafunc{\norm{ x^j - \Pi_{\X(y^\prime)} \left[x^j - \nabla f(x^j)\right] } },
	\end{equation}
	where $\sigmafunc{ \cdot } \geq 0$. The sequence $\{ f(x^j) \}$ is therefore nonincreasing. Moreover, \cref{eq:x_j+1-x_j} implies that
	\begin{equation}
	\label{eq:abs_value_f_xj}
	\left|f(x^{j+1}) - f(x^j)\right| \geq \sigma\left(\norm{ x^j - \Pi_{\X(y^\prime)} \left[x^j - \nabla f(x^j)\right] }\right).
	\end{equation}
	By \cref{ass:level_set_compact}, $\{ f(x^j) \}$ is lower bounded. Therefore, recalling that $\{ f(x^j) \}$ is nonincreasing, we get that $\{ f(x^j) \}$ converges, which implies that 
	\[ \left|f(x^{j+1}) - f(x^j)\right| \to 0. \]
	By \cref{eq:abs_value_f_xj}, we get that ${\sigma\left( \norm{x^j - \Pi_{\X(y^\prime)} \left[x^j - \nabla f(x^j)\right]} \right) \to 0}$, and, by the properties of $\sigmafunc{ \cdot }$, we finally get that $\norm{x^j - \Pi_{\X(y^\prime)} \left[x^j - \nabla f(x^j)\right]} \to 0$, and this contradicts \cref{eq: contradiction}.
\end{proof}
The next proposition shows some properties of the sequences generated by the algorithm, which will play an important role in the subsequent analysis.
\begin{proposition}
	\label{prop:sequences_properties}
	Let $\{ (x^k, y^k) \}$, $\{\mu_k\}$ and $\{ \eta_k\}$ be the sequences produced by the algorithm. Then:
	\begin{enumerate}[label=\textnormal{(\roman*)}]
		\item \label{seq_prop_point_1} the sequence $\{ f(x^k) \}$ is nonincreasing and convergent;
		\item \label{seq_prop_point_2} the sequence $\{ (x^k, y^k) \}$ is bounded;
		\item \label{seq_prop_point_3} the set $K_u = \{k \mid \eta_k < \eta_{k-1} \}$ of unsuccessful iterates is infinite;
		\item \label{seq_prop_point_4} $\lim_{k \to \infty} \mu_k = 0$;
		\item \label{seq_prop_point_5} $\lim_{k \to \infty} \eta_k = 0$;
		\item \label{seq_prop_point_6} $\lim_{k \to \infty} \norm{ x^k - \Pi_{\X(y^k)} \left[x^k - \nabla f(x^k)\right] } = 0$.
	\end{enumerate}
\end{proposition}

\begin{proof}
	\begin{enumerate}[label=\textnormal{(\roman*)}]
		\item The instructions of the algorithm and \cref{prop:property_A} imply that $\{ f(x^k) \}$ is nonincreasing, and \cref{ass:level_set_compact} implies that $\{ f(x^k) \}$ is lower bounded. Hence, $\{ f(x^k) \}$ converges.
		\item The instructions of the algorithm imply that each point $(x^k, y^k)$ belongs to the level set $\mathcal{L}(x^0, y^0)$, which is compact by \cref{ass:level_set_compact}. Therefore, $\{ (x^k, y^k) \}$ is bounded.
		\item Suppose that $K_u$ is finite. Then there exists $\bar{k} > 0$ such that all iterates satisfying $k > \bar{k}$ are successful, i.e.,
		\[ f(x^k) \leq f(x^{k-1}) - \eta_{k-1}, \]
		and $\eta_k = \eta_{k-1} = \eta > 0$ for all $k \geq \bar{k}$. Since $\eta > 0$, this implies that $\{f(x^k)\}$ diverges to $-\infty$, in contradiction with \cref{seq_prop_point_1}.
		\item Since, for all $k$, $\mu_{k+1} = \delta \mu_k$, where $\delta \in (0, 1)$, the claim holds.
		\item If $k \in K_u$, then $\eta_{k+1} = \theta \eta_k$, where $\theta \in (0,1)$. Since $K_u$ is infinite and $\eta_{k+1}=\eta_k$ if $k\notin K_u$, the claim holds.
		\item By \cref{prop:property_A}, we have that
		\[ f(\tilde{x}^k) - f(x^k) \leq -\sigmafunc{ \norm{ x^k - \Pi_{\X(y^k)} \left[x^k - \nabla f(x^k)\right]} }. \]
		By the instructions of the algorithm, $f(x^{k+1}) \leq f(\tilde{x}^k)$, and so we can write
		\[ f(x^{k+1}) - f(x^k) \leq -\sigmafunc{ \norm{ x^k - \Pi_{\X(y^k)} \left[x^k - \nabla f(x^k)\right]} }, \]
		i.e.,
		\[ 	\left| f(x^{k+1}) - f(x^k) \right| \geq \sigmafunc{ \norm{ x^k - \Pi_{\X(y^k)} \left[x^k - \nabla f(x^k)\right]} }. \]
		Since $\{ f(x^k) \}$ converges, we get that ${\sigmafunc{ \norm{ x^k - \Pi_{\X(y^k)} \left[x^k - \nabla f(x^k)\right]} } \to 0}$. By the properties of $\sigmafunc{ \cdot }$, we get that $\norm{ x^k - \Pi_{\X(y^k)} \left[x^k - \nabla f(x^k)\right]} \to 0$.
	\end{enumerate}
\end{proof}

Before stating the main theorem of this section, it is useful to summarize some theoretical properties of the subsequence $\{(x^k, y^k)\}_{K_u}$ of the unsuccessful iterates. As the proof shows, the next proposition follows easily from the theoretical results we have shown above.
\begin{proposition}
	\label{prop:sequence_K_u}
	Let $\{(x^k, y^k)\}$ be the sequence of iterates generated by \cref{alg:MISO}, and let $K_u = \{k \mid \eta_k < \eta_{k-1} \}$. Then:
	\begin{enumerate}[label=\textnormal{(\roman*)}]
		\item $\{(x^k, y^k)\}_{K_u}$ admits accumulation points;
		\item for any accumulation point $(x^*, y^*)$ of the sequence $\{(x^k, y^k)\}_{K_u}$, every \linebreak ${(\hat{x}, \hat{y}) \in \N(x^*, y^*)}$ is an accumulation point of a sequence $\{(\hat{x}^k, \hat{y}^k)\}_{K_u}$ where $(\hat{x}^k, \hat{y}^k) \in \N(x^k, y^k)$.
	\end{enumerate}
\end{proposition}
\begin{proof}
	\begin{enumerate}[label=\textnormal{(\roman*)}]
		\item By \cref{prop:sequences_properties}, \cref{seq_prop_point_2}, $\{(x^k, y^k)\}$ is bounded. Therefore, $\{(x^k, y^k)\}_{K_u}$ is also bounded, and so it admits accumulation points.
		\item \cref{prop:discrete_neighborhoods_X} implies that every $(\hat{x}, \hat{y}) \in \N(x^*, y^*)$ is an accumulation point of a sequence $\{(\hat{x}^k, \hat{y}^k)\}_{K_u}$, where $(\hat{x}^k, \hat{y}^k) \in \N(x^k, y^k)$.
	\end{enumerate}
\end{proof}

We can now prove the main theoretical result of this section.
\begin{theorem}\label{theo_convergence}
	Let $\{(x^k, y^k)\}$ be the sequence generated by \cref{alg:MISO}. Every accumulation point $(x^*,y^*)$ of $\{(x^k, y^k)\}_{K_u}$ is such that $x^*$ is an $\N$-stationary point of problem \cref{prob: cc_prob}.
\end{theorem}
\begin{proof}
	Let $(x^*, y^*)$ be an accumulation point of $\{(x^k, y^k)\}_{K_u}$. We must show that conditions (i)-(iii) of \cref{def: stationary} are satisfied.
	\begin{enumerate}[label=\textnormal{(\roman*)}]
		\item From the instructions of \cref{alg:MISO} the iterates $(x^k,y^k)$ belong to the set ${\cal L}(x^0,y^0)$, which is closed from \cref{ass:level_set_compact}. Any limit point $(x^*,y^*)$ belongs to ${\cal L}(x^0,y^0)$ and is thus feasible for problem \eqref{prob:mi_prob}.
		\item The result follows from \cref{prop:sequences_properties}, \cref{seq_prop_point_6}.
		\item Since $K_u$ is an infinite subset of unsuccessful iterations, recalling that $x^k = \tilde{x}^{k-1}$, $y^k = y^{k-1}$, and setting $\hat{x}^k = \hat{x}^{k-1}$, $\hat{y}^k = \hat{y}^{k-1}$ for all $(\hat{x}^{k-1}, \hat{y}^{k-1}) \in \N(\tilde{x}^{k-1}, y^{k-1})$, the test at Step 3 fails at iteration $k$, and therefore
		\begin{equation*}
		f(\hat{x}^k) > f(x^k) - \eta_{k-1}
		\end{equation*}
		for all $(\hat{x}^k, \hat{y}^k) \in \N(x^k, y^k)$. Since the sequence $\{f(x^k)\}$ is nonincreasing (\cref{prop:sequences_properties}, \cref{seq_prop_point_1}), we can write
		\begin{equation*}
		f(x^*) \leq f(x^k) < f(\hat{x}^k) + \eta_{k-1}.
		\end{equation*}
		for all $(\hat{x}^k, \hat{y}^k) \in \N(x^k, y^k)$. Taking limits, we get from \cref{prop:sequences_properties}, \cref{seq_prop_point_5}, \cref{prop:discrete_neighborhoods_no_X}, and by the continuity of $f$ that $f(x^*)\le f(\hat{x})$ for all $(\hat{x}, \hat{y}) \in \N(x^*, y^*)$.
		
		Now, note that \cref{seq_prop_point_1} of \cref{prop:sequences_properties} ensures the existence of $f^* \in \R$ satisfying
		\begin{equation}
		\label{eq:theo_1}
		\lim\limits_{k \to \infty} f(x^k) = f(x^*) = f^*.
		\end{equation}
		Consider any $(\hat{x}, \hat{y}) \in \N(x^*, y^*)$ such that
		\begin{equation}
		\label{eq:theo_2}
		f(\hat{x}) = f^*.
		\end{equation}
		\Cref{prop:sequence_K_u} implies that $(\hat{x}, \hat{y})$ is an accumulation point of a sequence $\{(\hat{x}^k, \hat{y}^k)\}_{K_u}$, where $(\hat{x}^k, \hat{y}^k) \in \N
		(x^k, y^k)$. Since $k \in K_u$, we have that $x^k = \tilde{x}^{k-1}$, $y^k = y^{k-1}$. Setting $\hat{x}^k = \hat{x}^{k-1}$, $\hat{y}^k = \hat{y}^{k-1}$ for all $(\hat{x}^{k-1}, \hat{y}^{k-1}) \in \N(\tilde{x}^{k-1}, y^{k-1})$, by \cref{eq:theo_1} and \cref{eq:theo_2} we get, for $k$ sufficiently large,
		\begin{equation*}
		f(\hat{x}^k) < f(x^k) + \xi.
		\end{equation*}
		Therefore, for such values of $k$, $(\hat{x}^k, \hat{y}^k) \in W_k$, and Steps 3.2-3.4 produce the points $x_k^2,\ldots,x_k^{j_k^*}$ (where $j_k^*$ is the finite number of iterations of Steps 2.2-2.4 until the test at Step 2.4 fails), which, by the instructions at Step 2.2 and by \cref{prop:property_A}, satisfy
		\begin{equation}
		\label{eq:theo_3}
		f(\hat{x}^k) \geq f(x_k^2) \geq \ldots \geq f(x_k^{j_k^*}).
		\end{equation}
		Since $k \in K_u$, Step 2.3 fails, and we can write
		\begin{equation}
		\label{eq:theo_4}
		f(x_k^{j_k^*}) > f(\tilde{x}^k) - \eta_k \geq f(x^k) - \eta_{k-1}.
		\end{equation}
		Moreover, as the sequence $\{(\hat{x}^k, \hat{y}^k)\}_{K_u}$ converges to the point $(\hat{x}, \hat{y})$, by \cref{eq:theo_1}, \cref{eq:theo_2}, \cref{eq:theo_3}, \cref{eq:theo_4}, and by \cref{seq_prop_point_5} of \cref{prop:sequences_properties}, we obtain
		\begin{equation*}
		f^* = \lim\limits_{k \to \infty, k \in K_u} f(\hat{x}^k) = \lim\limits_{k \to \infty, k \in K_u} f(x_k^2) = \lim\limits_{k \to \infty, k \in K_u} f(x^k) = f^*.
		\end{equation*}
		By \cref{prop:property_A}, we have that
		\begin{equation*}
		f(x_k^2) \leq f(\hat{x}^k) - \sigma\left(\norm{ \hat{x}^k - \Pi_{\X(\hat{y}^k)} \left[\hat{x}^k - \nabla f(\hat{x}^k)\right]}\right),
		\end{equation*}
		which can be rewritten as
		\begin{equation*}
		\left| f(x_k^2) - f(\hat{x}^k) \right| \geq \sigma\left(\norm{ \hat{x}^k - \Pi_{\X(\hat{y}^k)} \left[\hat{x}^k - \nabla f(\hat{x}^k)\right]}\right).
		\end{equation*}
		Taking limits for $k \to \infty, k \in K_u$, we finally get
		\begin{equation*}
		\norm{ \hat{x} - \Pi_{\X(\hat{y})} \left[\hat{x} - \nabla f(\hat{x})\right]} = 0,
		\end{equation*}
		and the claim holds.
	\end{enumerate}
\end{proof}

In \cite{Beck2016}, the concept of \textit{basic feasibility} (BF) introduced in \cite{Beck2013} is extended to problem \eqref{prob: cc_prob}:
\begin{definition}\label{def_BF}
	A feasible point $x^*$ of problem \eqref{prob: cc_prob}  is referred to as basic feasible if, for any super support set  $J$, letting $y_J\in\{0,1\}^n$ such that $y_i=0$ if $i\in J$ and $y_i=1$ otherwise, there exists $L>0$ such that $$x^*=\Pi_{\X(y_J)}\left(x^*+d\right),$$
	where $d_i = -\frac{1}{L}\nabla_i f(x^*)$ if $i\in J$ and $d_i=0$ otherwise.
\end{definition}

Note that BF stationarity requires that, for any $y_J$ defining a super support set, $x^*=\Pi_{\X(y_J)}[x^*+d]$, where $d_J=-\frac{1}{L}\nabla_J f(x^*)$ and $d_{\bar{J}}=0,$ whereas the condition in \cref{stat_1} requires $x^*=\Pi_{\X(y_J)}[x^*-\nabla f(x^*)]$. In fact, in the case of our problem the two conditions are equivalent, as we show below.

\begin{lemma}\label{BF_proj}
	Let $y\in\cal Y$ and  $x^*\in\X(y)$. Then $x^*$ satisfies $$x^*=\Pi_{\X(y)}(x^*+d),$$ where $d_{I_0(y)}=-\frac{1}{L}\nabla_J f(x^*)$ and $d_{I_1(y)}=0,$  if and only if it satisfies $$x^*=\Pi_{\X(y)}(x^*-\nabla f(x^*)).$$
\end{lemma}
\begin{proof}
	Let us consider
	$$\hat{x} = \Pi_{\X(y)}[x^*-\nabla f(x^*)],\qquad \tilde{x} = \Pi_{\X(y)}[x^*+d].$$
	Let us denote $\hat{x}^p = x^*-\nabla f(x^k)$ and $ \tilde{x}^p = x^*+d$ Since both $\hat{x}$ and $\tilde{x}$ belong to $\X(y)$, we have $\hat{x}_{I_1(y)}=0$ and $\tilde{x}_{I_1(y)}=0$. From the well-known properties of the projection operator on a convex set, we get:
	$$(\hat{x}-\hat{x}^p)^{\top}(\hat{x}-x)\le 0\quad\forall\,x\in X(y),$$
	$$(\tilde{x}-\tilde{x}^p)^{\top}(\tilde{x}-x)\le 0\quad\forall\,x\in X(y),$$
	hence
	$$(\hat{x}-\hat{x}^p)^{\top}(\hat{x}-\tilde{x})\le 0,\qquad (\tilde{x}-\tilde{x}^p)^{\top}(\tilde{x}-\hat{x})\le 0.$$
	Taking into account that $\tilde{x}_{I_1(y)}=\hat{x}_{I_1(y)} = 0$ and $\hat{x}_{I_0(y)}^p=\tilde{x}_{I_0(y)}^p=x^*-\nabla_{I_0(y)} f(x^*)$, we get
	$$(\hat{x}_{I_0(y)}-\hat{x}^p_{I_0(y)})^{\top}(\hat{x}_{I_0(y)}-\tilde{x}_{I_0(y)})\le 0,\qquad (\tilde{x}_{I_0(y)}-\hat{x}^p_{I_0(y)})^{\top}(\tilde{x}_{I_0(y)}-\hat{x}_{I_0(y)})\le 0,$$
	i.e.,
	$$\|\hat{x}_{I_0(y)}\|^2 - \tilde{x}_{I_0(y)}^{\top}\hat{x}_{I_0(y)} - \hat{x}_{I_0(y)}^{\top}\hat{x}_{I_0(y)}^p + \tilde{x}_{I_0(y)}^{\top}\hat{x}_{I_0(y)}^p\le0,$$
	and
	$$\|\tilde{x}_{I_0(y)}\|^2-\tilde{x}_{I_0(y)}^{\top}\hat{x}_{I_0(y)}-\tilde{x}_{I_0(y)}^{\top}\hat{x}_{I_0(y)}^p+\hat{x}_{I_0(y)}^{\top}\hat{x}_J^p\le 0$$
	Summing up the two inequalities, we get
	$$\|\hat{x}_{I_0(y)}\|^2+\|\tilde{x}_{I_0(y)}\|^2-2\hat{x}_{I_0(y)}^{\top}\tilde{x}_{I_0(y)}\le0,$$
	i.e.,
	$$\|\hat{x}_{I_0(y)}-\tilde{x}_{I_0(y)}\|\le0,$$
	from which we obtain $\tilde{x}_{I_0(y)}=\hat{x}_{I_0(y)}$ and hence $\hat{x} = \tilde{x}.$
\end{proof}

We can hence show that, provided that $\N_\rho$ is employed as neighborhood in \ref{alg:MISO}, with a sufficiently large value of $\rho$, the SNS procedure converges to basic feasible solutions.
\begin{theorem}\label{BF_stat}
	Let $\{(x^k, y^k)\}$ 
	be the sequence of iterates generated by \cref{alg:MISO} equipped with $\N_\rho$ as neighborhood and $\cal A^*$ the set of the accumulation points of the sequence $\{(x^k, y^k)\}_{K_u}$ of unsuccessful iterates. If $\rho\geq2(s-\delta^*)$, in the definition of the set ${\cal N}_\rho(x,y)$, and $\delta^*=\min\{\|x^*\|_0\ |\ (x^*,y^*)\in {\cal A^*}\}$, 
	then given a point $(x^*,y^*)\in {\cal A}^*$, $x^*$ is basic feasible for problem \eqref{prob: cc_prob}.
\end{theorem}
\begin{proof}
	Let $J\subset \{1,\ldots ,n\}$ be any super support set for $x^*$, and consider the vector $\hat y$ such that
	$
	\hat y_j=1\quad \forall j\notin J
	$
	and zero otherwise.
	As $|J|=s$,
	we have $e^\top\hat y=n-s$, and, taking into account that $i\notin J$ implies $x_i^*=0$ and
	$i\in J$ implies $\hat y_i=0$,
	it follows
	$$
	x_i^*\hat y_i=0\quad i=1,\ldots ,n.
	$$
	Then, we have $I_1(x^*)\subseteq I_0(\hat y)$ and
	$(x^*,\hat y) \in \bar {\cal N}(x^*)\subseteq {\cal N}_\rho(x^*, y^*)$, where we used \cref{prel_lemma}.
	By taking into account \cref{theo_convergence}, we finally get that $x^*$ is an $\mathcal{N}_\rho$-stationary point of problem \cref{prob: cc_prob} and that it is also a stationary point of
	\begin{equation*}
	\begin{aligned}
	\min \;& f(x) \\
	\text{s.t. } & x \in \mathcal{X}(\hat y),
	\end{aligned}
	\end{equation*}
	that is 
	$$x^*=\Pi_{\X(\hat y)}(x^*-\nabla f(x^*)).$$
	Then, by \cref{BF_proj}, recalling that $\hat y_i=0$ if and only if $i\in J$, we obtain that $x^*$ is basic feasible.
\end{proof}

\section{Convergence results under constraint qualifications}
\label{sec:convergenceKKT}
In this section, we show that, under constraint qualifications and by choosing suitable neighborhoods, it is possible to state convergence results similar to those considered in important works of the related literature \cite{Burdakov2016,Lu2013}. Here, we assume that $X = \{x\in\R^n\mid g(x)\le 0, \;h(x) = 0\}$, where $h_i$, $i=1,\ldots,p$ are affine functions and $g_i$, $i=1,\ldots,m$, are convex functions.
First we state the following assumption which
implicitly involves constraint qualifications. 
\begin{assumption}\label{KKT}
Given $\bar y \in \Y$ and $\bar x\in \mathcal{X}(\bar y)$, 
we have that $\bar x$ is a stationary point of problem \eqref{subprob_cont} if and only if there exist multipliers $\lambda\in\mathbb{R}^m$, $\mu\in\mathbb{R}^p$ and $\gamma\in\mathbb{R}^n$ such that
	\begin{equation*}
\begin{array}{r}
\displaystyle	\nabla f(\bar x) + \sum_{i=1}^m \lambda_i \nabla g_i(\bar x) +
\sum_{i=1}^p \mu_i \nabla h_i(\bar x) + \sum_{i=1}^n \gamma_i e_i = 0,\\
\lambda_i \ge 0,  \ \lambda_i g_i(\bar x) = 0, \
\forall i = 1, \ldots, m,\\
\gamma_i = 0, \ \forall\; i  \text{ such that } \bar y_i=0.
\end{array}
\end{equation*}
\end{assumption}
The above assumption states that 
$\bar x$ is a stationary point of problem \eqref{subprob_cont} if and only if it is a KKT point of the following problem
\begin{equation*}
	\begin{aligned}
	\min_{x}\; & f(x)  \\
	\text{s.t. } & h_i(x) = 0,	 \quad \forall i=1,\ldots,p,\\
	& g_i(x) \leq 0, \quad \forall i=1,\ldots,m,\\
	& x_i \bar y_i= 0, \quad\; \forall i=1,\ldots,n, \\
	\end{aligned}
\end{equation*}
which can be equivalenty rewritten as follows
\begin{equation*}
\begin{aligned}
\min_{x}\; & f(x)  \\
\text{s.t. } & h_i(x) = 0,	 \quad \forall i=1,\ldots,p,\\
& g_i(x) \leq 0, \quad \forall i=1,\ldots,m,\\
&x_i = 0, 	\qquad\; \forall i: \bar y_i=1.
\end{aligned}
\end{equation*}
\begin{remark}
As shown in \cref{sec:appendix}, \cref{KKT} holds when, e.g.,   the functions $g_i$  are strongly convex with constant $\mu_i>0$, for $i=1,\ldots ,m$, the functions $h_j$, for $j=1,\ldots ,p$ are affine, and 
some Cardinality Constraint-Constraint Qualification (CC-CQ) is satisfied. 
For instance, a standard CC-CQ is
the Cardinality Constraint- Linear Independence Constraint Qualification (CC-LICQ), requiring that  the gradients 
\begin{align*}
	\nabla g_i(\bar x)\qquad & \text{ for all } i: g_i(\bar{x}) = 0\\
	\nabla h_i(\bar{x})\qquad& \text{ for all } i=1,\ldots,p\\
	e_i\qquad&\text{ for all } i:\bar{y}_i = 1
\end{align*}
are linearly independent.
\end{remark}
\par\medskip\noindent
From \cref{theo_convergence} and \cref{KKT} we immediately get the following result.
\begin{theorem}\label{S_stat}
	Let $\{(x^k, y^k)\}$ be the sequence generated by \cref{alg:MISO}. Every accumulation point $(x^*,y^*)$ of the sequence of unsuccessful iterates $\{(x^k, y^k)\}_{K_u}$ is such that there exist multipliers $\lambda\in\mathbb{R}^m$, $\mu\in\mathbb{R}^p$ and $\gamma\in\mathbb{R}^n$ such that
	\begin{equation}
	\label{eq:S-stat}
\begin{array}{r}
\displaystyle	\nabla f(x^*) + \sum_{i=1}^m \lambda_i \nabla g_i(x^*) +
\sum_{i=1}^p \mu_i \nabla h_i(x^*) + \sum_{i=1}^n \gamma_i e_i = 0,\\
\lambda_i \ge 0,  \ \lambda_i g_i(x^*) = 0, \
\forall i = 1, \ldots, m,\\
\gamma_i = 0, \ \forall\; i  \text{ such that } y_i^*=0.
\end{array}
\end{equation}
\end{theorem}

\begin{remark}
	Condition \eqref{eq:S-stat} is the $S$-\textit{stationarity} concept introduced in \cite{Burdakov2016}. Basically, the limit points of the sequence $\{(x^k,y^k)\}_{K_u}$ produced by \cref{alg:MISO} are always guaranteed to be $S$-stationary.
	This implies, by the results in \cite{Burdakov2016}, that $x^*$ is also Mordukhovich-stationary for problem \eqref{prob: cc_prob}. In fact, under \cref{KKT}, it is easy to see that $\mathcal{N}$-stationarity is a stronger condition than $M$-stationarity, from points (i)-(ii) of \cref{def: stationary}.
\end{remark}

\par\medskip\noindent
In order to state stronger convergence results, we need to use suitable neighborhoods (e.g., $\mathcal{N}_\rho$ with a sufficiently large value of $\rho$) in the algorithm.

\begin{theorem}\label{LZ_stat}
	Let $\{(x^k, y^k)\}$ 
	be the sequence generated by \cref{alg:MISO} equipped with $\N_\rho$ as neighborhood and $\cal A^*$ the set of the accumulation points of the sequence $\{(x^k, y^k)\}_{K_u}$ of unsuccessful iterates. If $\rho\geq2(s-\delta^*)$, in the definition of the set ${\cal N}_\rho(x,y)$, and $\delta^*=\min\{\|x^*\|_0\ |\ (x^*,y^*)\in {\cal A^*}\}$, 
	then given a point $(x^*,y^*)\in {\cal A}^*$ and for every super support set $J\subset \{1,\ldots ,n\}$, we have that there exist multipliers $\lambda\in\mathbb{R}^m$, $\mu\in\mathbb{R}^p$ and $\gamma\in\mathbb{R}^n$ such that
	\begin{equation}\label{thesis}
\begin{array}{r}
\displaystyle	\nabla f(x^*) + \sum_{i=1}^m \lambda_i \nabla g_i(x^*) +
\sum_{i=1}^p \mu_i \nabla h_i(x^*) + \sum_{i=1}^n \gamma_i e_i = 0,\\
\lambda_i \ge 0,  \ \lambda_i g_i(x^*) = 0, \
\forall i = 1, \ldots, m,\\
\gamma_i = 0, \ \forall\; i  \in J.
\end{array}
\end{equation}
\end{theorem}
\begin{proof}
Let $J\subset \{1,\ldots ,n\}$ be any super support set for $x^*$, and consider the vector $\hat y$ such that
$
\hat y_j=1\quad \forall j\notin J
$
and zero otherwise.
As $|J|=s$,
we have $e^\top\hat y=n-s$, and, taking into account that $i\notin J$ implies $x_i^*=0$ and
$i\in J$ implies $\hat y_i=0$,
it follows
$$
x_i^*\hat y_i=0\quad i=1,\ldots ,n.
$$
Then, we have $I_1(x^*)\subseteq I_0(\hat y)$ and
$(x^*,\hat y) \in \bar {\cal N}(x^*)\subseteq {\cal N}_\rho(x^*, y^*)$, where we used \cref{prel_lemma}.
By taking into account \cref{theo_convergence}, we finally get that $x^*$ is an $\mathcal{N}_\rho$-stationary point of problem \cref{prob: cc_prob} and that it is also a stationary point of
\begin{equation*}
			\begin{aligned}
			\min \;& f(x) \\
			\text{s.t. } & x \in \mathcal{X}(\hat y).
			\end{aligned}
		\end{equation*}
Then, by \cref{KKT}, recalling that $\hat y_i=0$ if and only if $i\in J$, we obtain that \eqref{thesis} holds.
\end{proof}

\begin{remark}
Condition \eqref{thesis} is the necessary optimality condition first defined in \cite{Lu2013}. It is interesting to note that the Penalty Decomposition algorithm proposed in the referenced work in fact is not guaranteed to converge to a point satisfying such conditions, that are guaranteed to hold only if the limit point has full support. In the general case, the PD method generates points satisfying \eqref{thesis} for at least one super support set. Our SNS algorithm would have the same exact convergence results if we used the neighborhood
$$\mathcal{N}(x^k,y^k)=\{(x,y)\mid x=x^k,\;e^{\top}y=n-s,\;y_ix_i^k=0\,\forall\,i\}.$$
The above neighborhood basically checks all the super support sets at the current iterate $x^k$, but it does not satisfy the continuity \cref{asmp:discrete_neighborhood}, hence failing to guarantee that condition \eqref{thesis} is satisfied by all super support sets at the limit point.
\end{remark}

\section{Numerical Experiments}
\label{sec:experiments}
From a computational point of view, we are particularly interested in studying two relevant aspects. Specifically, here we want to: 
\begin{itemize}
    \item analyze the benefits and the costs of increasing the size of the neighborhood;
    \item assess the performance of the proposed approach, compared to the the Greedy Sparse-Simplex (GSS) method proposed in \cite{Beck2013} and the Penalty Decomposition (PD) approach \cite{Lu2013}.
\end{itemize}

To these aims, we considered the problem of sparse logistic regression, where the objective function
is continuously differentiable and convex, but the solution of the problem for a fixed support set requires the adoption of an iterative method. Note that we preferred to consider a problem without other constraints in addition to the sparsity one, in order to simplify the analysis of the behavior of the proposed algorithm.

The problem of \textit{sparse logistic regression}  \cite{hastie2009elements} has important
applications, for instance, in machine learning \cite{bach2012optimization,weston2003use}.
Given a dataset having $N$ samples $\{z^1,\ldots ,z^N\}$, with $n$ features and $N$ corresponding labels
$\{t_1,\ldots, t_N\}$
belonging to 
$\{-1,1\}$,
the problem of sparse maximum likelihood estimation of a logistic regression model can be formulated as follows
\begin{equation}
    \label{eq:sparse_logistic}
    \begin{aligned}
        \min_w\;&L(w) =\sum_{i=1}^{N} \log\left(1+\exp\left(-t_i(w^{\top}z^i)\right)\right) \\\text{ s.t. }&\|w\|_0\le s.
    \end{aligned}
\end{equation}

The benchmark for this experiment is made up of problems of the form \eqref{eq:sparse_logistic}, obtained as described hereafter. We employed 6 binary classification datasets, listed in \cref{tab:logistic_datasets}. All the datasets are from the UCI Machine Learning Repository \cite{Dua2019}. 
For each dataset, we removed data points with missing variables; moreover, we one-hot encoded the categorical variables and standardized the other ones to zero mean and unit standard deviation.
For every dataset, we chose different values of $s$, as specified later in this section. 

\begin{table}[!htb]
	\centering
	\caption{List of datasets used for experiments on sparse logistic regression.}
	\begin{tabular}{lccc}
		\hline
		\textbf{Dataset} & $\boldsymbol{N}$ & $\boldsymbol{n}$ & \textbf{Abbreviation}\\
		\hline
		Heart (Statlog) & 270 & 25 & \texttt{heart} \\
		Breast Cancer Wisconsin (Prognostic) & 194 & 33 & \texttt{breast} \\
		QSAR Biodegradation & 1055 & 41 & \texttt{biodeg} \\
		SPECTF Heart & 267 & 44 & \texttt{spectf} \\
		Spambase & 4601 & 57 & \texttt{spam} \\
		Adult a2a & 2265 & 123 & \texttt{a2a} \\
		\hline
	\end{tabular}
	\label{tab:logistic_datasets}
\end{table}

\subsection{Implementation details}
Algorithms SNS, PD and GSS have been implemented in Python 3.7, mainly exploiting libraries \texttt{numpy} and \texttt{scipy}.
The convex subproblems of both PD and GSS have been solved up to global optimality by using the L-BFGS algorithm (in the implementation from \cite{liu1989limited}, provided by \texttt{scipy}). We also employed L-BFGS for the local optimization steps in SNS. 
All algorithms start from the feasible initial point $x^0=0\in\mathbb{R}^n$. For the PD algorithm, we set the starting penalty parameter to 1 and its growth rate to 1.05. The algorithm stops when $\|x^k-y^k\|<0.0001$, as suggested in \cite{Lu2013}. AS for the GSS, we stop the algorithm as soon as $\|x^{k+1}-x^k\|\le0.0001$.

Concerning our proposed \cref{alg:MISO}, the parameters have been set as follows:
\begin{itemize}
\item $\xi = 10^{3}$,
\item $\theta = 0.5$,
\item $\eta_0 = 10^{-5}$.
\end{itemize}
For what concerns $\mu_0$ and $\delta$, we actually keep the value of $\mu$ fixed to $10^{-6}$. We again employ the stopping criterion $\|x^{k+1}-x^k\|\le0.0001$.

For all the algorithms, we have also set a time limit of $10^4$ seconds. All the experiments have been carried out on an Intel(R) Xeon E5-2430 v2 @2.50GHz CPU machine with 6 physical cores (12 threads) and 16 GB RAM.

As benchmark for our experiments, we considered 18 problems, obtained from the 6 datasets in \cref{tab:logistic_datasets} and setting $s$ to 3, 5 and 8 in \eqref{eq:sparse_logistic}. For SNS and GSS we consider the computational time employed to find the best solution. We take into account four versions of \cref{alg:MISO}, with neighborhood radius $\rho\in\{1,2,3,4\}$.

In \cref{fig:pp} the performance profiles \cite{Dolan2002} w.r.t.\ the objective function values and the runtimes (intended as the time to find the best solution) attained by the different algorithms are shown. We do not report the runtime profile of SNS(1) since it is much faster than all the other methods and thus would dominate the plot, making it poorly informative. We can however note that unfortunately its speed is outweighed by the very poor quality of the solutions.  We can observe that increasing the size of the neighborhood consistently leads to higher quality solutions, even though the computational cost grows. We can see that SNS (with a sufficiently large neighborhood) has better performances than the other algorithms known from the literature; in particular, while the neighborhood radius $\rho=1$ only allows to perform forward selection, with poor outcomes, $\rho\ge 2$ makes swap operations possible, with a significant impact on the exploration capabilities. The GSS has worse quality performance than SNS(2), which is reasonable, since its move set is actually smaller and optimization is always carried out w.r.t.\ a single variable and not the entire active set. However, it proved to also be slower than the SNS, mostly because of two reasons: it always tries all feasible moves, not necessarily accepting the first one that provides an objective decrease, and it requires many more iterations to converge, since it considers one variable at a time. Finally, the PD method appears not to be competitive from both points of view: it is slow at converging to a feasible point and it has substantially no global optimization features that could guide to globally good solutions.

\begin{figure}[htbp]
	\centering
	\subfloat[objective value]{\includegraphics[width=0.48\textwidth]{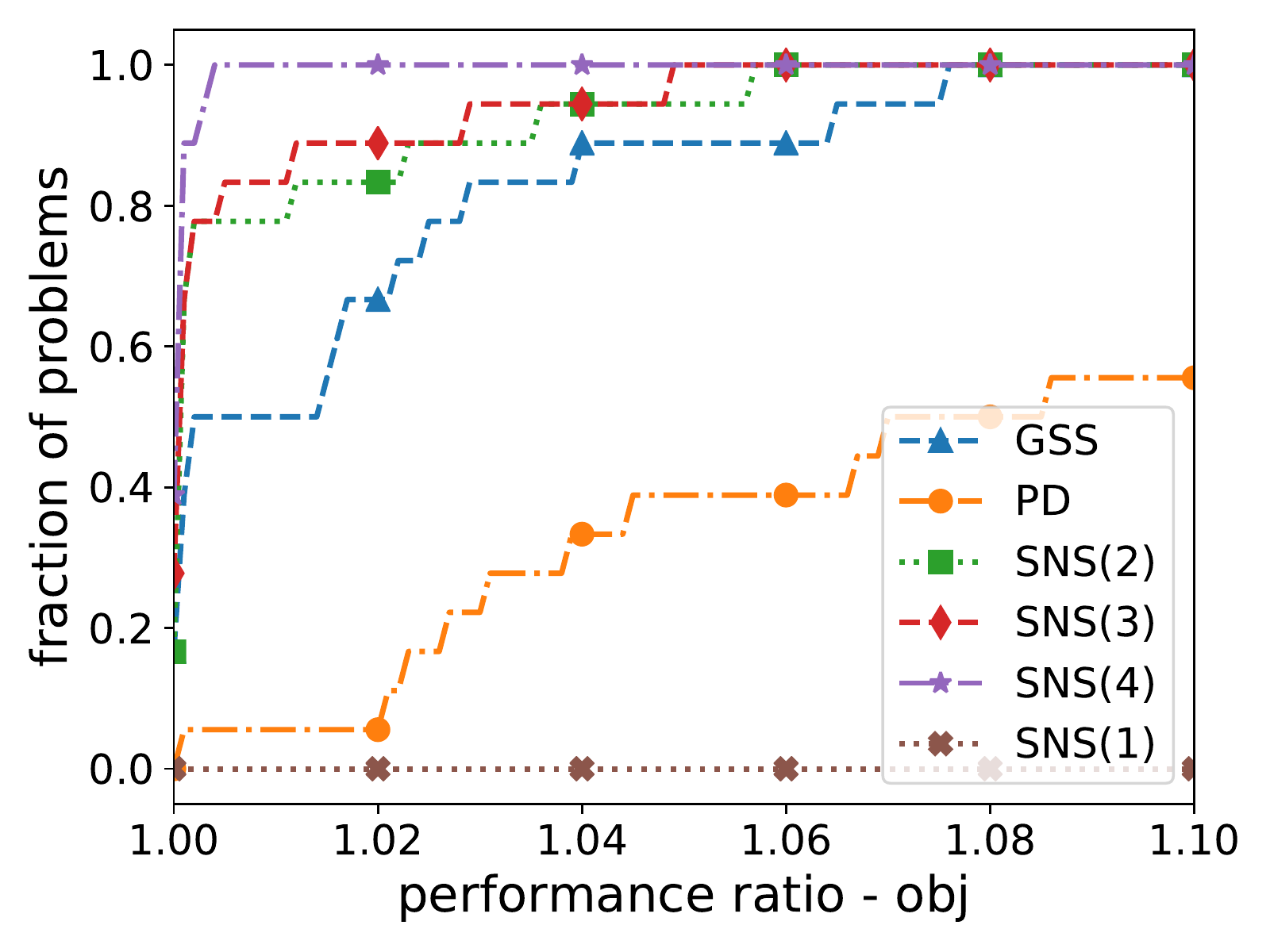}}
	\subfloat[time]{\includegraphics[width=0.48\textwidth]{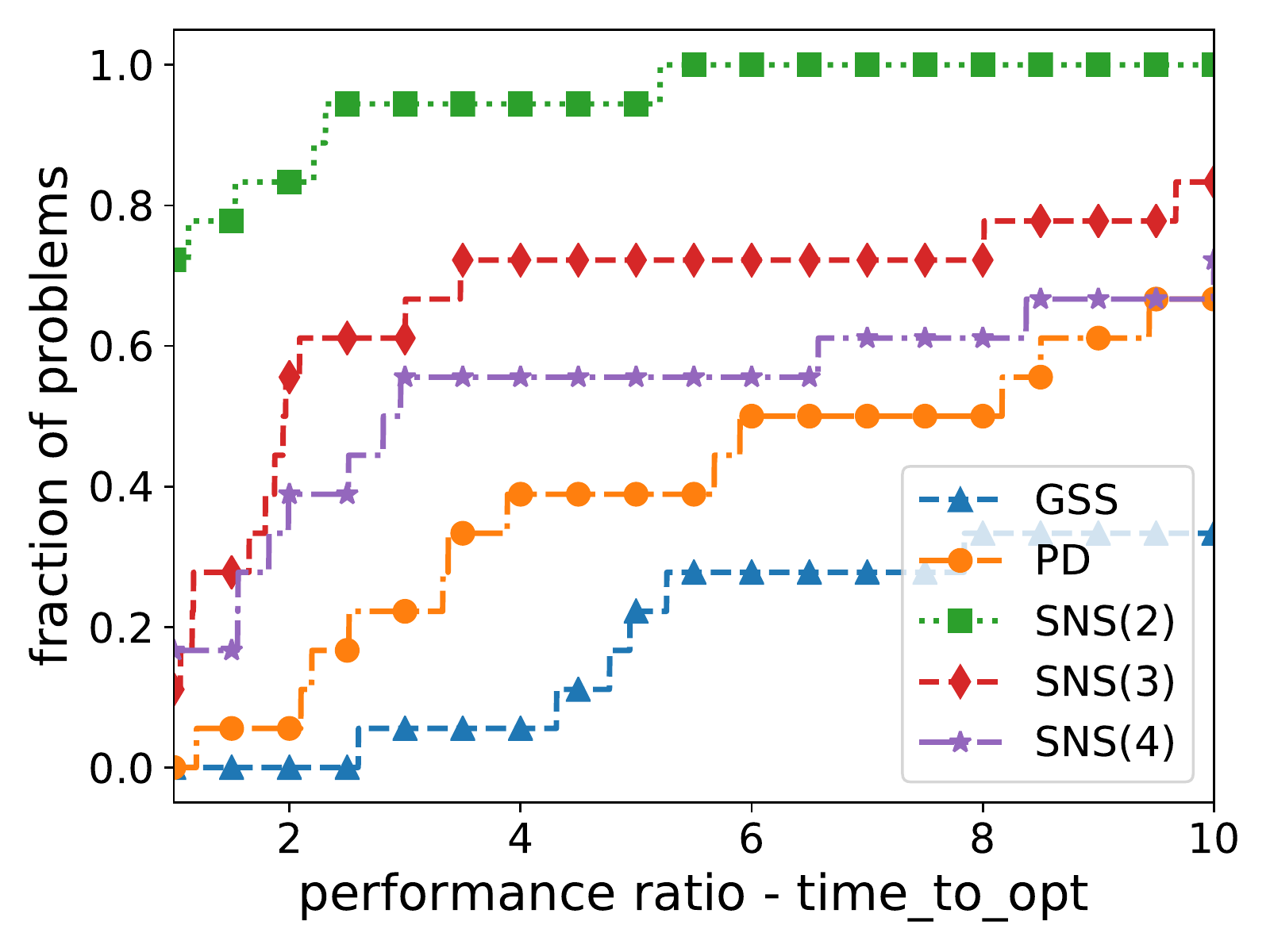}}
	\caption{Performance profiles for the considered algorithms on 18 sparse logistic regression problems.}
	\label{fig:pp}
\end{figure}

It is interesting to remark how considering larger neighborhoods appears to be particularly useful in problems where the sparsity constraint is less strict and thus combinatorially more challenging. As an example, we show the runtime-objective tradeoff for the \texttt{breast}, \texttt{spam} and \texttt{a2a} problems for $s=3$ and $s=8$ in \cref{fig:tradeoffs}. We can observe that for $s=3$, SNS finds good, similar solutions for either $\rho=2,3$ or $4$, with a similar computational cost. On the other hand, as $s$ grows to 8, using $\rho=4$ allows to significantly improve the quality of the solution without a significant increase in terms of runtime. 

\begin{figure}[htbp]
	\centering
	\subfloat[\texttt{breast} - $s=3$]{\includegraphics[width=0.48\textwidth]{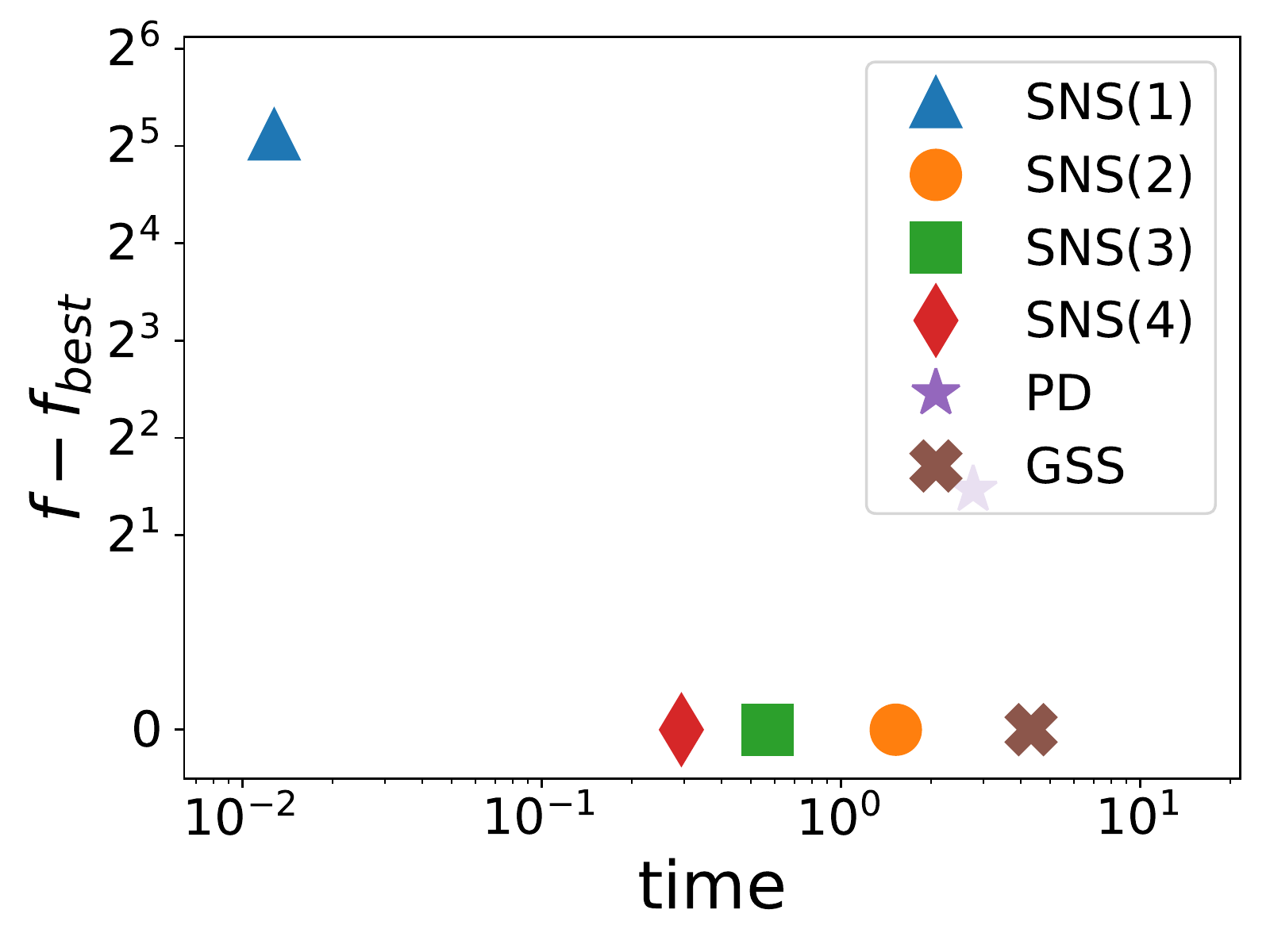}}
	\subfloat[\texttt{breast} - $s=8$]{\includegraphics[width=0.48\textwidth]{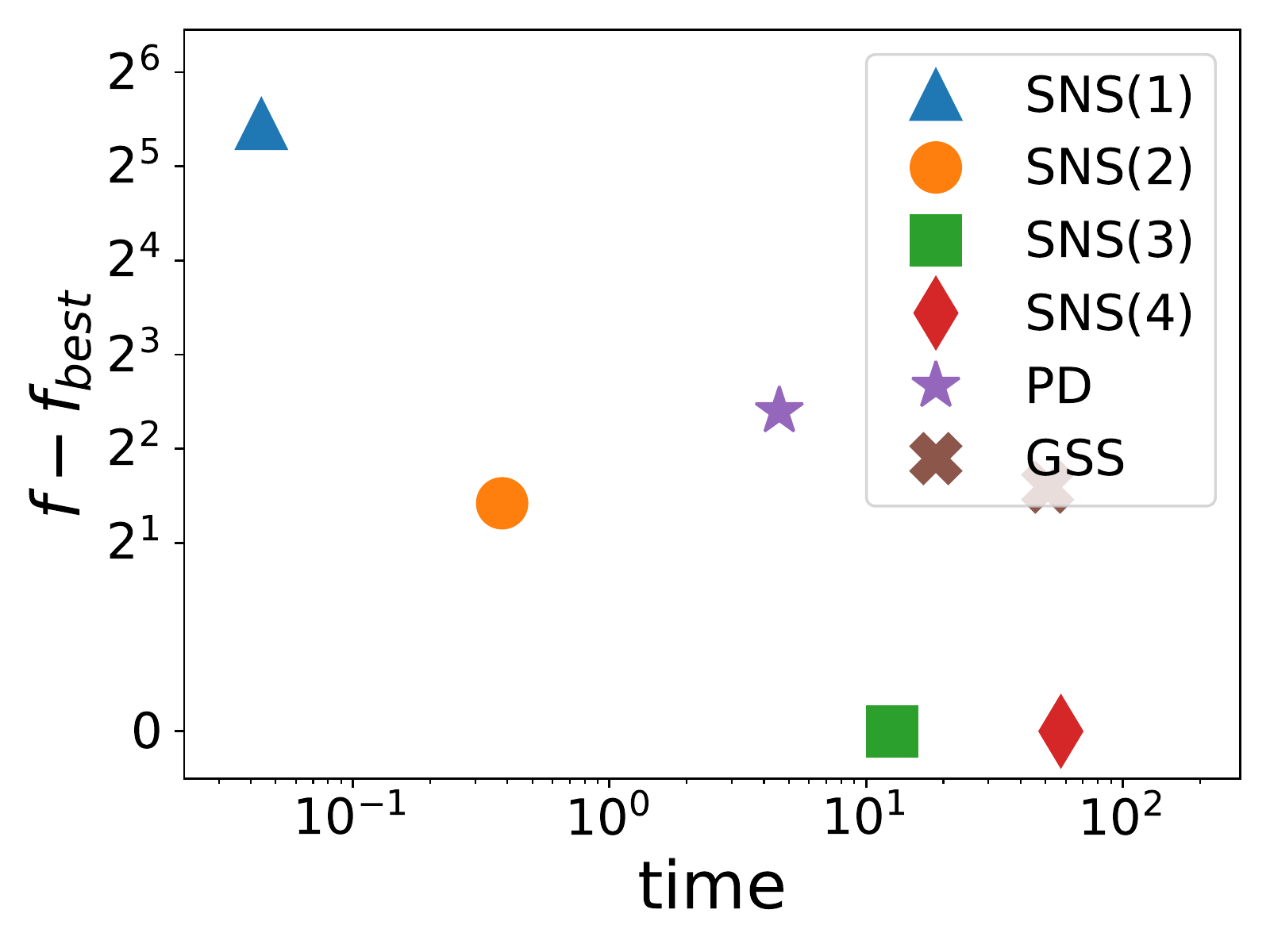}}
	
	\subfloat[\texttt{spam} - $s=3$]{\includegraphics[width=0.48\textwidth]{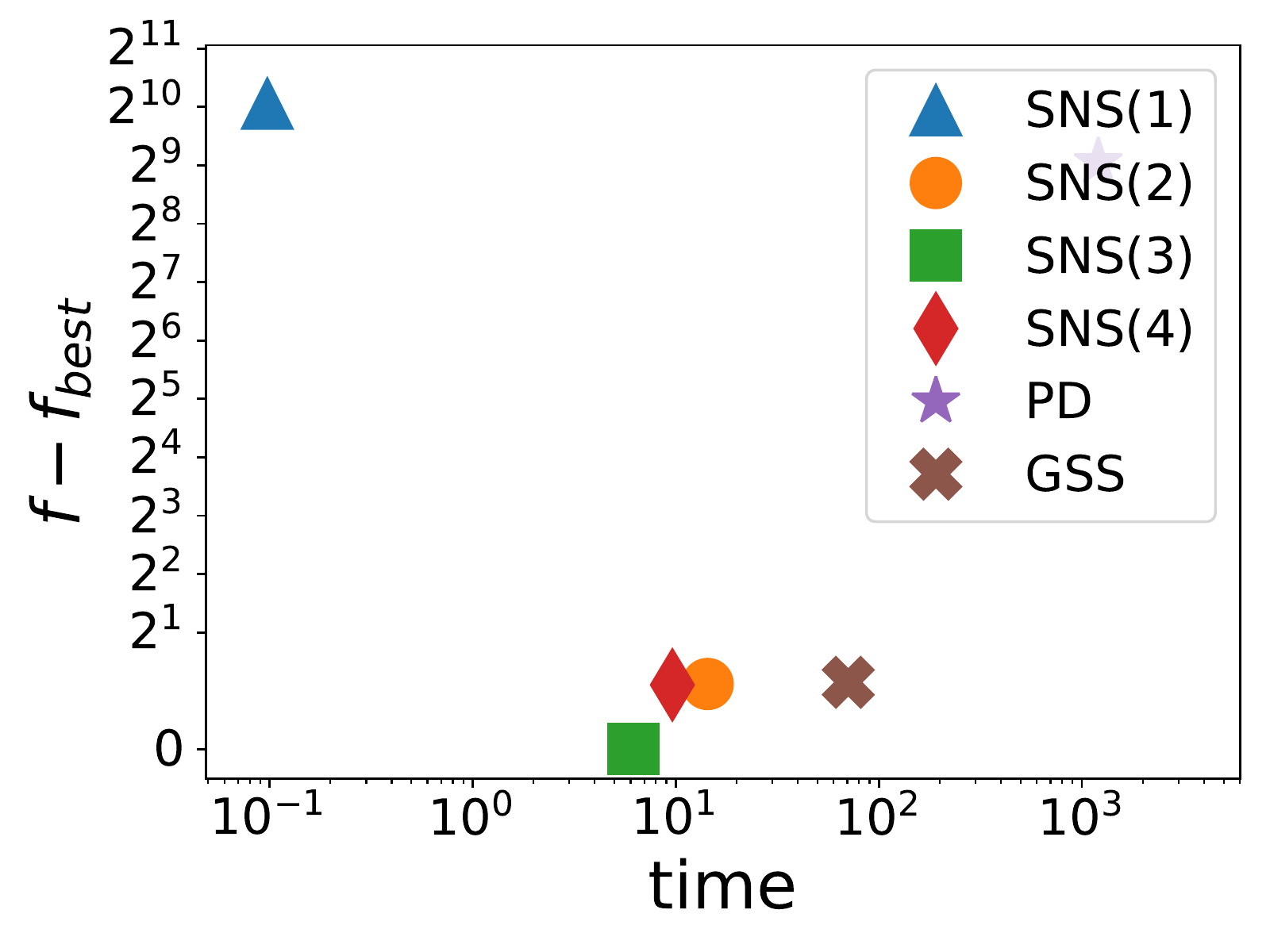}}
	\subfloat[\texttt{spam} - $s=8$]{\includegraphics[width=0.48\textwidth]{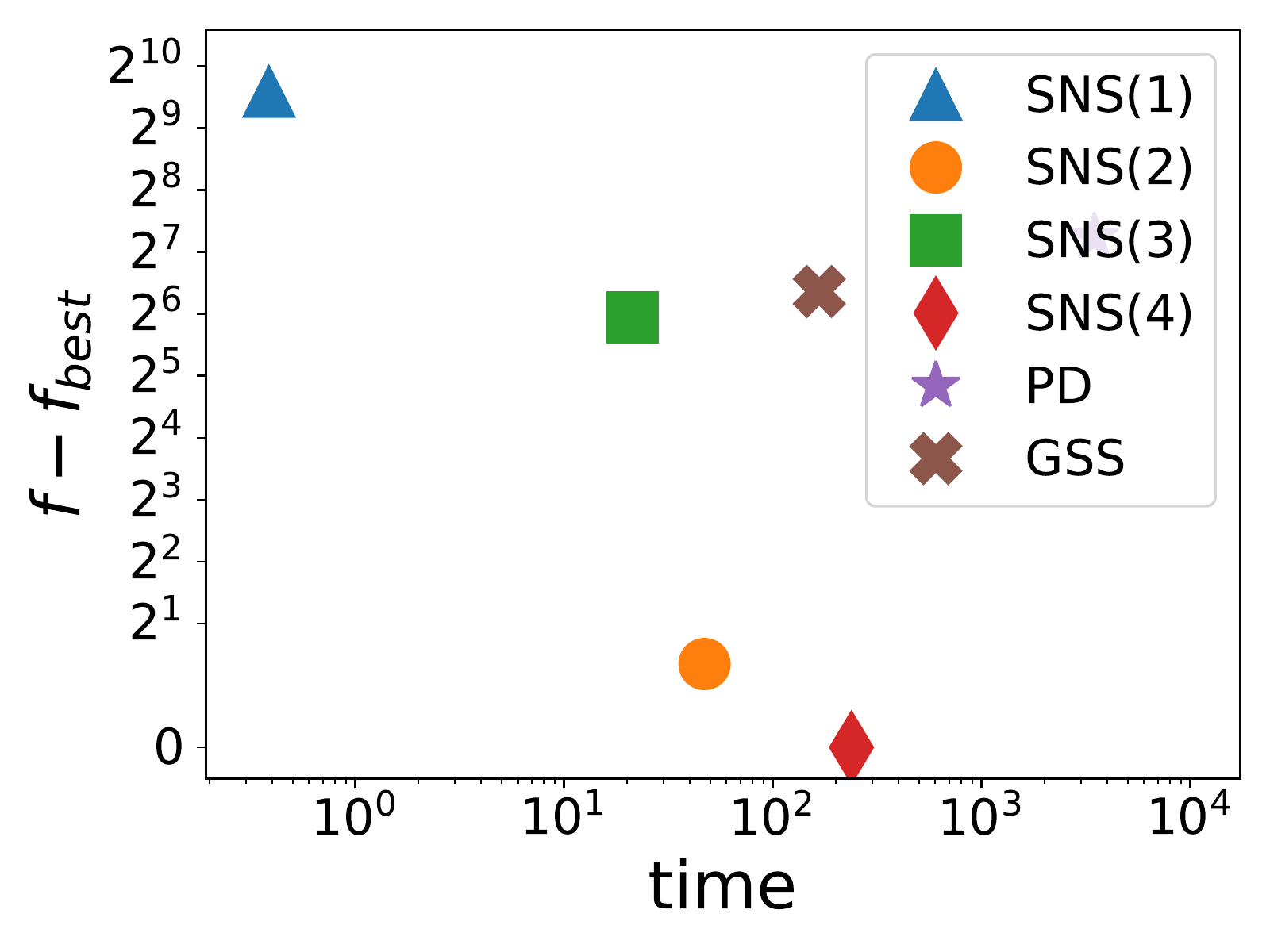}}
	
	\subfloat[\texttt{a2a} - $s=3$]{\includegraphics[width=0.48\textwidth]{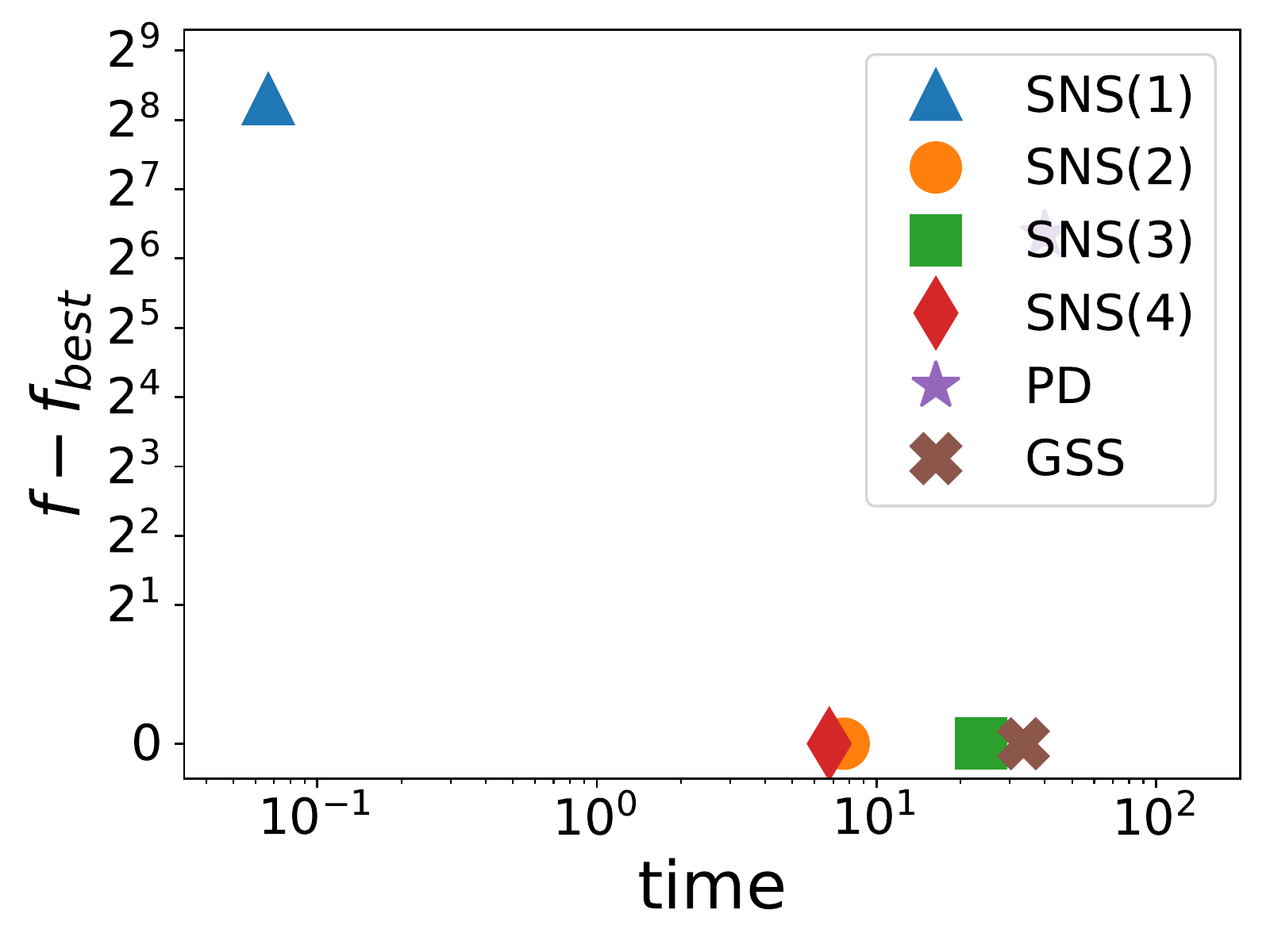}}
	\subfloat[\texttt{a2a} - $s=8$]{\includegraphics[width=0.48\textwidth]{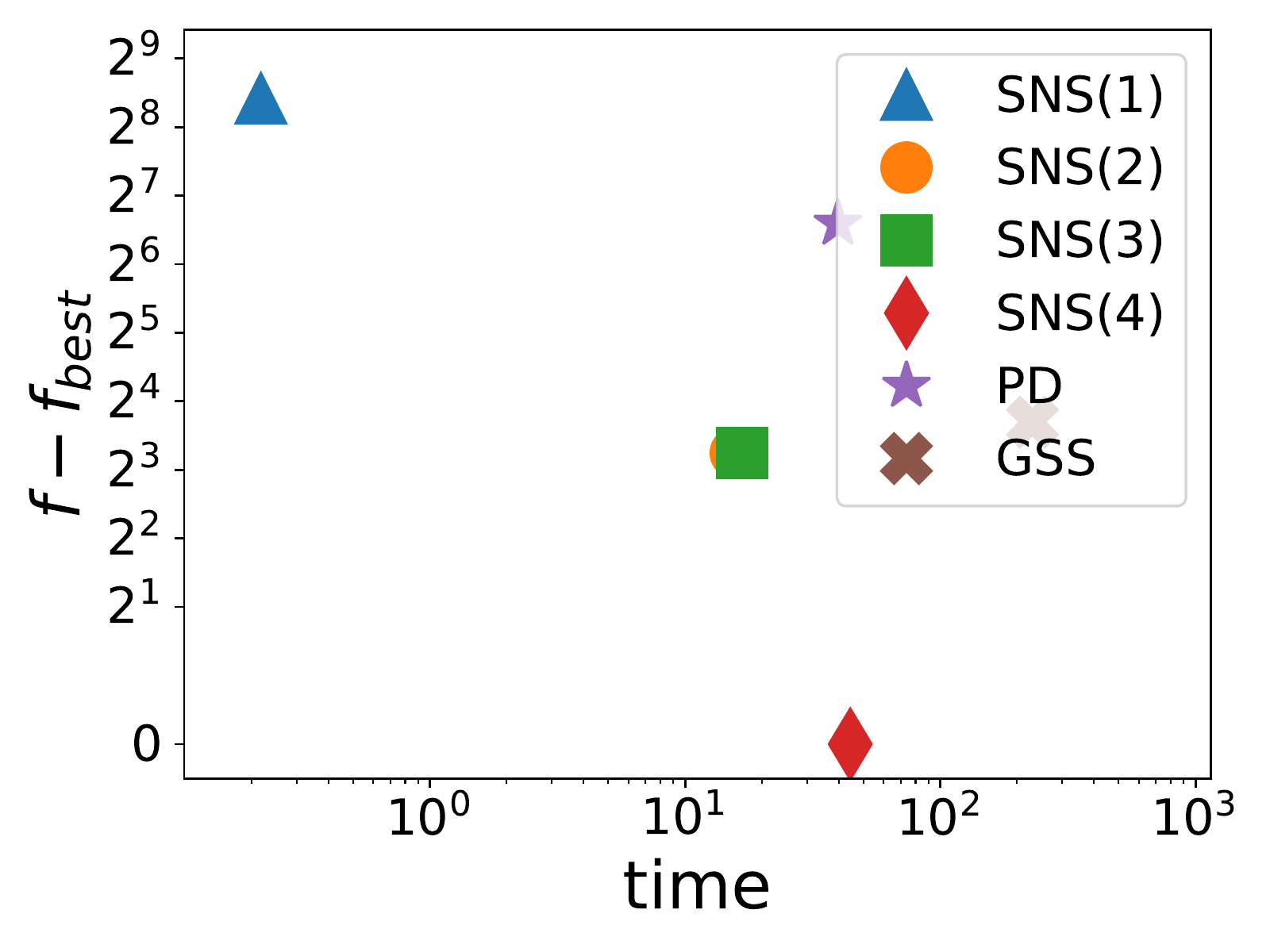}}
	\caption{Quality/cost trade-off for the algorithms on sparse logistic regression problems from datasets \texttt{breast}, \texttt{spam} and \texttt{a2a}.}
	\label{fig:tradeoffs}
\end{figure}

\FloatBarrier

\section{Conclusions}
\label{sec:conclusions}
In this paper we have analyzed sparsity constrained optimization problems. For this class of problems, we have defined a necessary optimality condition, namely, $\N$-stationarity, exploiting the concept of discrete neighborhood associated with a well-known mixed integer equivalent reformulation, that allows to take into account potentially advantageous changes on the set of active variables.

We have afterwards proposed an algorithmic framework to tackle the  family of problems under analysis. Our SNS method alternates continuous local search steps and neighborhood exploration steps; the algorithm is then proved to produce a sequence of iterates whose cluster points are $\N$-stationary. Moreover, we proved that, by suitably employing a tailored neighborhood, the limit points also satisfy other optimality conditions from the literature, based on both gradient projection and Lagrange multipliers, thus providing stronger optimality guarantees than other state-of-the-art approaches.

Finally, we studied the features and the benefits of our proposed procedure from a computational perspective. Specifically, we compared the performance of the SNS as the size of the neighborhood increases, observing that using wider neighborhoods consistently provides higher quality solutions with a reasonable increase of the computational cost, especially when the required cardinality  is not that small. Moreover, when comparing SNS with the Penalty Decomposition method and the Greedy Sparse-Simplex method, we observed that our method has higher exploration capability, thus getting a nice match between theory and practice, and  it is affordable in terms of computational cost, being even faster than the other considered methods.

\appendix
\section{On the relationship between stationarity conditions and KKT conditions}
\label{sec:appendix}
Consider the continuous optimization problem
\begin{equation}
\label{eq:convex_prob}
\begin{aligned}
\min_{x}\;& f(x)\\\text{s.t. }& x\in X,
\end{aligned}
\end{equation}
where $X=\{x\in\mathbb{R}^n\mid h(x) = 0,\;g(x)\le 0\}$ is a convex set ($h_i$, $i=1,\ldots,p$ are affine functions, $g_i$, $i=1,\ldots,m$, are convex functions). We assume $f$ and $g$ to be continuously differentiable; $h$ is differentiable, being affine.

\begin{definition}
	A point $x^*\in X$ is a \textit{stationary point} for problem \eqref{eq:convex_prob} if, for any direction $d$ feasible at $x^*$, we have
	\begin{equation*}
	\nabla f(x^*)^{\top}d\ge 0. 
	\end{equation*}
\end{definition}
It can be shown that a point $x^*$ is stationary for problem \eqref{eq:convex_prob} if and only if
\begin{equation}
\label{eq:proj-stat}
x^* = \Pi_X[x^*-\nabla f(x^*)],
\end{equation}
where $\Pi_X$ denotes the orthogonal projection operator.
Stationarity is a necessary condition of optimality for problem \eqref{eq:convex_prob}.
It is possible to show that a point satisfying the KKT conditions is always a stationary point.
Viceversa is true by stronger assumptions on the set of feasible directions.

\begin{proposition}
	\label{prop:appendix}
	Let $x^*\in X$ satisfy KKT conditions for problem \eqref{eq:convex_prob}. Then, $x^*$ is stationary for problem \eqref{eq:convex_prob}.
\end{proposition}
\begin{proof}
	Assume $x^*$ satisfies KKT conditions with multipliers $\lambda$ and $\mu$. Let $d$ be any feasible direction at $x^*$. Since $X$ is convex, we know that:
	\begin{gather}
		\label{eq:app_h}
		\nabla h_i(x^*)^{\top}d=0\quad \forall i=1,\ldots,p,\\
		\label{eq:app_g}
		\nabla g_i(x^*)^{\top}d\le 0 \quad \forall i:g_i(x^*)=0.
	\end{gather}
	Moreover, from KKT conditions we know that 
	\begin{gather}
		\label{eq:app_compl}
		\lambda_i=0 \quad \forall\, i:g_i(x^*)<0.
	\end{gather}		

We know that
$$\nabla f(x^*) + \sum_{i=1}^{m}\lambda_i\nabla g_i(x^*) + \sum_{i=1}^{m}\mu_i\nabla h_i(x^*)=0,$$
hence
$$\left(\nabla f(x^*) + \sum_{i=1}^{m}\lambda_i\nabla g_i(x^*) + \sum_{i=1}^{p}\mu_i\nabla h_i(x^*)=0\right)^{\top}d=0,$$
and then
$$\nabla f(x^*)^{\top}d + \sum_{i=1}^{m}\lambda_i\nabla g_i(x^*)^{\top}d + \sum_{i=1}^{m}\mu_i\nabla h_i(x^*)^{\top}d=0.$$
From equations \eqref{eq:app_h} and \eqref{eq:app_compl}, we get
$$\nabla f(x^*)^{\top}d + \sum_{i:g_i(x^*)=0}^{}\lambda_i\nabla g_i(x^*)^{\top}d=0,$$
thus, recalling \eqref{eq:app_g} and $\lambda\ge0$,
$$\nabla f(x^*)^{\top}d =- \sum_{i:g_i(x^*)=0}^{}\lambda_i\nabla g_i(x^*)^{\top}d\ge 0.$$
Since $d$ is an arbitrary feasible direction, we get the thesis.
\end{proof}
\begin{proposition}
	\label{prop2:appendix}
	Let $x^*\in X$ be a stationary point for problem \eqref{eq:convex_prob}.
	Assume that one of the following conditions holds:
	\begin{itemize}
 \item [(i)] the set of feasible direction $D(x^*)$ is such that
$$
  D(x^*)=\{d\in \mathbb{R}^n:\nabla g_i(x^*)^{\top}d\le 0, \ i\in I(x^*), \nabla h_i(x^*)^{\top}d=0, i=1,\ldots ,p\}
$$
 \item [(ii)] the set of feasible direction $D(x^*)$ is such that
 \begin{gather*}
 	D(x^*)=\{d\in \mathbb{R}^n\mid\nabla g_i(x^*)^{\top}d<0, \ i: g_i(x^*) = 0,\ \nabla h_j(x^*)^{\top}d=0, j=1,\ldots ,p\},
 \end{gather*}
and a constraint qualification holds.
\end{itemize}
Then, $x^*$ is a KKT point.
\end{proposition}
\begin{proof}
Assertion (i). Let $x^*$ be a stationary point. Then, there does not exist a direction $d\in D(x^*)$ such that
$$
\nabla f(x^*)^{\top}d<0.
$$
This implies that the system
$$
\begin{array}{ccc}
\nabla f(x^*)^{\top}d&<0&\\
\nabla g_i(x^*)^{\top}d & \le 0 & \ \ i: g_i(x^*) = 0\\
\nabla h_i(x^*)^{\top}d&\le 0 & i=1,\ldots ,p\\
-\nabla h_i(x^*)^{\top}d&\le 0 & i=1,\ldots ,p\\
\end{array}
$$
does not admit solution. By Farkas' Lemma we get the thesis.
\par\medskip\noindent
Assertion (ii). Let $x^*$ be a stationary point. Then, there does not exist a direction $d\in D(x^*)$ such that
$$
\nabla f(x^*)^{\top}d<0.
$$
This implies that the system
$$
\begin{array}{ccc}
\nabla f(x^*)^{\top}d&<0&\\
\nabla g_i(x^*)^{\top}d & < 0 & \ \ i: g_i(x^*) = 0\\
\nabla h_i(x^*)^{\top}d&= 0 & i=1,\ldots ,p\\
\end{array}
$$
does not admit solution. By Motzkin’s theorem  we get that $x^*$ satisfies the Fritz-John conditions and hence, by assuming a constraint qualification, the thesis is proved.
\end{proof}
\par\medskip\noindent
Condition (i) of \cref{prop2:appendix} holds if the functions $g_i$, $i=1,\ldots ,m$, $h_j$, $j=1,\ldots ,p$ are affine.
\par\medskip\noindent
Condition (ii) of \cref{prop2:appendix} 
holds by assuming that the convex functions $g_i$, for $i=1,\ldots ,m$ are such that
\begin{equation}\label{d1}
g_i(x+td)\ge g_i(x)+t\nabla g_i(x)^{\top}d+\frac{1}{2}\gamma t^2\|d\|^2
\end{equation}
with $\gamma >0$. Indeed, in this case it is easy to see that a direction $d$ is a feasible direction at $x^*$ if and only if
$$
\nabla g_i(x^*)^{\top}d<0\quad \ \ i: g_i(x^*) = 0
\quad\quad
\nabla h_j(x^*)^{\top}d=0\quad i=1,\ldots ,p
$$
\par\medskip\noindent
Condition \eqref{d1} is satisfied by assuming that the functions $g_i$ are twice continuosly differentiable and the Hessian matrix is positive definite.
\par\medskip\noindent
Condition \eqref{d1} holds also for continuously
differentiable functions $g_i$ assuming that they are {\it strongly convex with constant $\mu_i>0$}, i.e., that for $i=1,\ldots ,m$ the functions
$$
g_i(y)\ge
g_i(x)+\nabla g_i(x)^\top(y-x)+\frac{\mu_i}{2}\|y-x\|^2, \quad \forall\ x,y.
$$

\bibliographystyle{plain}
\bibliography{bibliography}
\end{document}